\documentclass[twoside, a4paper, 10pt]{article}
\usepackage[english]{babel}
\usepackage[utf8]{inputenc}

\usepackage[nottoc]{tocbibind}
\usepackage{amsthm}
\usepackage{bm}
\usepackage{hyperref}
\usepackage{enumerate}
\usepackage{graphicx, psfrag, amscd, amssymb,amsmath,amsthm, amsfonts}
\usepackage{mathrsfs}

\usepackage[dvipsnames]{xcolor}

\usepackage[frak=pxtx]{mathalpha}

\usepackage{tikz-cd}
\usepackage{upgreek}

\newtheorem{theorem}{Theorem}[section]
\newtheorem{corollary}[theorem]{Corollary}
\newtheorem*{intro_theorem}{Theorem}
\newtheorem{lemma}[theorem]{Lemma}
\newtheorem*{question}{Question}

\theoremstyle{definition}
\newtheorem{definition}[theorem]{Definition}
\newtheorem*{intro_definition}{Definition}
\newtheorem{example}[theorem]{Example}

\theoremstyle{remark}
\newtheorem{remark}[theorem]{Remark}

\DeclareMathOperator{\spt}{spt}
\DeclareMathOperator{\im}{im}

\DeclareMathOperator{\domain}{dmn}
\DeclareMathOperator{\tangent}{Tan}

\DeclareMathOperator{\homomorphism}{Hom}
\DeclareMathOperator{\divergence}{div}

\DeclareMathOperator{\closure}{Clos}
\DeclareMathOperator{\interior}{Int}

\DeclareMathOperator{\der}{D}           

\DeclareMathOperator{\with}{:} 
\DeclareMathOperator{\without}{\sim} 

\DeclareMathOperator{\dimension}{dim}
\DeclareMathOperator{\restrict}{\llcorner}

\DeclareMathOperator{\lipschitz}{Lip}

\newcommand{\ud}{\ensuremath{\,\mathrm{d}}}

\begin{document}

\title{Young functions on varifolds. Part I. Functional analytic foundations}
\author{Hsin-Chuang Chou}
\maketitle

\begin{abstract}
    The series of papers is devoted to the study of convergence for pairs of surfaces and smooth functions thereon. We model such pairs with varifolds and multiple-valued functions to capture their limits. In the present paper, we study Young functions, a measure-theoretic approach to multiple-valued functions, and the graph measures associated with pairs of measures (in particular, varifolds) and Young functions. This setting allows us to model the convergence of pairs of surfaces and functions thereon via the weak convergence of their associated graph measures, and a compactness theorem follows immediately. As a prerequisite for the concepts of differentiability for Young functions in the upcoming papers, we introduce and investigate several test function spaces.
\end{abstract}



\quad \\
\noindent
\textbf{Acknowledgements.}
The author would like to thank his PhD supervisor, Prof.\ Ulrich Menne, and his PhD co-supervisor, Dr.\ Nicolau Sarquis Aiex, for reading the paper and giving comments and suggestions. Also, the author would like to thank Mr.\ Yu-Tong Liu for consultation on functional analysis. The author was supported by the NTNU ``Scholarship Pilot Program of the Ministry of Science and Technology to Subsidize Colleges and Universities in the Cultivation of Outstanding Doctoral Students''.

%%%%%%%%%%%%%%%%%%%%%%%%%%%%%%%%%%%%%%%%%%%%%%%%%%%%%%%%%%%%%%%%%%%%%%%%%%%%%%%%%%%%%%%%%%%%%%%%%%%%%%%%%%%
\section{Introduction}
%%%%%%%%%%%%%%%%%%%%%%%%%%%%%%%%%%%%%%%%%%%%%%%%%%%%%%%%%%%%%%%%%%%%%%%%%%%%%%%%%%%%%%%%%%%%%%%%%%%%%%%%%%%

This is the first of a series of papers aiming to answer the following questions.
\begin{question}
    In a Euclidean space, consider a sequence of smooth surfaces and smooth functions thereon such that over every compact subset the area of the surfaces and the integrals of mean curvature, the functions themselves, and their derivatives are uniformly bounded. What is the limit of these pairs? What can we say about the smoothness of the limit?
\end{question}
Such questions typically arise when solving geometric variational problems, and those functions carry the geometric information about the surfaces. We shall begin with the first question. Note that varifolds (see Allard \cite{Allard72_MR307015}) are a natural candidate to model the convergence of surfaces to retain information on their mean curvature in the limit process. With this model, we allow surfaces to have multiplicity, and the following example shows that we should also allow functions to have multiple values. Consider two parallel lines approaching each other, and on each of them, the functions have constant values $1$ and $-1$, respectively; in this case, the expected limit should consist of a line of multiplicity $2$ and a constant two-valued function thereon. Almgren \cite{Almgren2000_MR1777737} has developed (later extended by De Lellis and Spadaro \cite{DLS11_MR2663735}) the theory of $Q$-valued functions in Euclidean spaces, where $Q$ is a fixed positive integer, to study the regularity of minimizers of Plateau problems in higher codimensions. However, this theory is not quite sufficient for our setting because, in general, the behavior of varifold convergence is more complicated, the number of values of the limit function varies from one point to another, and the functions can oscillate such that the values of the limit function become diffuse. Therefore, we introduce \textit{Young functions} (later or see \ref{definition: Young function}) to model multiple-valued functions. Consequently, to answer the second question, we should develop the concepts of differentiability for Young functions on varifolds in advance.

The purpose of the present paper is to provide functional analytic foundations for later topics; more precisely, the concept of Young functions and the convergence of pairs of measures and Young functions. We also study some specific test function spaces on which the concepts of differentiability of Young functions rely. The upcoming papers will then contribute to, firstly, the study of differentiability for Young functions and the compatibility of our theory with that of Almgren's $Q$-valued functions; secondly, the compactness properties of pairs of varifolds and ``differentiable'' Young functions, and therefore the answer to the second question.

\textbf{Description of the results.} 
To present the results, we shall first introduce the concept of Young functions.

\begin{intro_definition}[see \ref{definition: Young function}]
    Suppose $X$ and $Y$ are locally compact Hausdorff spaces and $\mu$ is a Radon measure over $X$. By a \textit{$\mu$ Young function $f$ of type $Y$}, we mean a $\mu$ measurable $\mathbf{P}(Y)$-valued function $f$, where $\mathbf{P}(Y)$ denotes the space of probability Radon measures over $Y$ endowed with the initial topology induced from the maps $\nu \mapsto \textstyle\int k \ud \nu$ for $\nu \in \mathbf{P}(Y)$ corresponding to continuous functions $k: Y \to \mathbf{R}$ with compact support.
\end{intro_definition}

Every $\mu$ measurable $Y$-valued function $g$ gives rise to an associated $\mu$ Young function $f$ of type $Y$ such that $f(x) = \boldsymbol{\updelta}_{g(x)}$ for $x \in \domain g$, see \ref{remark: single-valued Young function}; furthermore, we readily check that this statement still holds when $g$ is a Lebesgue measurable $Q$-valued function by taking a measurable selection of $g$, see \ref{example: Almgren's Q-valued function}. Compared with these examples, we emphasize that the values of Young functions can be diffuse; also, the number of values can vary from one point to another. Less trivial examples of Young functions can be obtained from the Young measures generated by a sequence of functions in $\mathbf{L}_\infty(\mathscr{L}^m)$ and from the disintegration of measures on products (for instance, the disintegration of varifolds); the first example also justifies the terminology of Young functions, see \ref{remark: Young measures} and \ref{example: varifold disintegration}.

To extend basic operations of single-valued functions to Young functions, we define typical operations on the space of probability measures, hence on the class of Young functions, such as pushforward, Cartesian product, and convolution, see \ref{definition: pushforward of Young function}, \ref{definition: product of Young functions}, and \ref{definition: convolution of Young functions}, respectively. When the Young functions are associated with $Y$-valued functions, these operations then correspond to the post-composition, join\footnote{Suppose $A$, $B$, and $C$ are sets and $f: A \to B$ and $g: A \to C$ are functions. By the join of $f$ and $g$, we mean the function $A \to B \times C$ given by $a \mapsto (f(a), g(a))$.}, and addition (provided $Y$ is a finite-dimensional Banach space) of $Y$-valued functions. Note that the notion of addition for multiple-valued functions is also new to the theory of $Q$-valued functions, but the addition of two $Q$-valued functions will be a $Q^2$-valued function, see \ref{remark: properties of convolution}\eqref{remark: properties of convolution3}.

To study the convergence of pairs of measures and Young functions, we adapt Hutchinson's graph measures (see \cite[4.3.1]{Hutchinson86_MR825628}) to our setting; in contrast to his concept of measure-function pairs in \cite[4.1.1]{Hutchinson86_MR825628}, we do not require any summability condition on Young functions.

\begin{intro_definition}[see \ref{definition: graph measure}]
    Suppose $X$ and $Y$ are locally compact Hausdorff spaces, $\mu$ is a Radon measure over $X$, and $f$ is a $\mu$ Young function of type $Y$. We define the \textit{graph measure} $\mathbf{Y}(\mu, f)$ associated with $\mu$ and $f$ to be the Radon measure over $X \times Y$ such that
    \[
        \textstyle\int \phi \ud \mathbf{Y}(\mu, f) = \textstyle\int \phi(x, y) \ud f(x) \, y \ud \mu \, x
    \]
    whenever $\phi: X \times Y \to \mathbf{R}$ is a continuous function with compact support.
\end{intro_definition}
When $X$ and $Y$ have countable bases and $f$ is the Young function associated with some $\mu$ measurable $Y$-valued function $g$, the graph measure $\mathbf{Y}(\mu, f)$ agrees with the pushforward of $\mu$ by the map $x \mapsto (x, g(x))$, see \ref{remark: comparison measure-function pair}. The convergence of pairs is then defined to be the convergence of their associated graph measures. For the class of Radon measures $\Gamma$ over $X \times Y$ such that $\Gamma(K \times Y) < \infty$ whenever $K$ is a compact subset of $X$, we study its compactness property and present a disintegration theorem, see \ref{theorem: compactness of graph measures} and \ref{theorem: disintegration}. Consequently, we have the following compactness theorem for pairs of varifolds and Young functions.

\begin{intro_theorem}[see \ref{theorem: compactness theorem for pairs of rectifiable varifolds and Young functions}]
    Suppose $k$, $m$, and $n$ are positive integers, $m \leq n$, $V_i$ is a sequence of $m$-dimensional varifolds in $\mathbf{R}^n$ that converges to $V$, and $f_i$ is a sequence of $\|V_i\|$ Young functions of type $\mathbf{R}^k$ such that
    \[
        \lim_{s \to \infty} \sup \{ \mathbf{Y}(\|V_i\|, f_i)(K \times (\mathbf{R}^k \without \mathbf{B}(0, s))) \with i = 1, 2, \dotsc \} = 0
    \]
    whenever $K$ is a compact subset of $\mathbf{R}^n$. Then, there exists a $\|V\|$ Young function $f$ of type $\mathbf{R}^k$ such that, possibly passing to a subsequence, we have
    \[
        \mathbf{Y}(\|V\|, f) = \lim_{i \to \infty} \mathbf{Y}(\|V_i\|, f_i).
    \]
    Moreover, if we require $V_i$ and $V$ to be rectifiable, then we can further conclude that $\mathbf{Y}(V, f \circ p) = \lim_{i \to \infty} \mathbf{Y}(V_i, f_i \circ p)$, where $p: \mathbf{R}^n \times \mathbf{G}(n, m) \to \mathbf{R}^n$ is the projection map.
\end{intro_theorem}

In preparation for investigating the concepts of differentiability for Young functions, we want to study the following test function spaces.
\begin{intro_definition}[see \ref{definition: function space E}]
    Suppose $Y$ is a finite-dimensional Banach space. Denote by $\mathbf{E}_s(Y)$ the vector space of all continuously differentiable functions $\gamma: Y \to \mathbf{R}$ such that $\gamma(0) = 0$ and $\spt \der \gamma \subset \mathbf{B}(0, s)$. The space $\mathbf{E}_s(Y)$ is endowed with the norm whose value equals $\sup \im \|\der \gamma\|$ at $\gamma \in \mathbf{E}_s(Y)$, and we endow the vector space $\mathbf{E}(Y) = \bigcup_{0 \leq s < \infty} \mathbf{E}_s(Y)$ with the locally convex final topology (see \ref{definition: locally convex final topology}) induced from the inclusions $\mathbf{E}_s(Y) \to \mathbf{E}(Y)$ corresponding to $0 \leq s < \infty$.
\end{intro_definition}
Recall that a function on Euclidean space is termed weakly differentiable if and only if it satisfies an integration-by-parts identity. We may define weak differentiability of $Y$-valued functions $g$ on varifolds in the same way, but as Menne pointed out in \cite[8.27]{Menne16a_MR3528825}, the chain rule fails for the class of such functions. To resolve this problem, we have to require the compositions $\gamma \circ g$ for $\gamma \in \mathbf{E}(Y)$, instead of merely the functions $g$ themselves, to satisfy an integration-by-parts identity, see \cite[8.3]{Menne16a_MR3528825}. When $f$ is the Young function associated with $g$, we have $(\gamma \circ g)(x) = \int \gamma \ud f(x)$ for $x \in \domain g$ and $\gamma \in \mathbf{E}(Y)$. Thus, the previous definition of weak differentiability suggests that it is expedient to view Young functions of type $Y$ as $\mathbf{E}(Y)^*$-valued functions, see \ref{remark: probability measures as dual elements}. With this in mind, we study the topological vector space structure of $\mathbf{E}(Y)$ and give a homeomorphic embedding of $\mathbf{E}(Y)$ into a space of continuous functions with compact support, see \ref{corollary: homeomorphic embedding E}. Moreover, to study the differentiability of Young functions on varifolds via their associated graph measures, we define the locally convex space $\mathbf{H}(U \times Y, \mathbf{R}^n)$ of certain test functions $\eta: U \times Y \to \mathbf{R}^n$, where $U$ is an open subset of $\mathbf{R}^n$; again, a homeomorphic embedding of $\mathbf{H}(U \times Y, \mathbf{R}^n)$ into a space of continuous functions with compact support is established, see \ref{theorem: homeomorphic embedding of H}.

As Lipschitz functions play an important role in the differentiation theory of functions on varifolds (cf. \cite{Menne16a_MR3528825} and \cite{Menne16b_MR3626845}), we expect their analogue for Young functions to play a similar role in our theory. To this end, we use $\mathbf{E}(Y)$ to define a pseudo-metric on $\mathbf{P}(Y)$ such that Dirac measures define an isometric embedding $Y \to \mathbf{P}(Y)$, which leads to a canonical definition of Lipschitz continuity for Young functions, see \ref{definition: pseudo-metric d}, \ref{remark: isometric embedding}, and \ref{definition: Lipschitz continuity Young function}. This pseudo-metric agrees with the $1$-Wasserstein metric on the subspace $\mathbf{P}_1(Y)$ of those $\mu \in \mathbf{P}(Y)$ such that $\int |y| \ud \mu \, y < \infty$, and the metric topology on $\mathbf{P}_1(Y)$ is characterized by the weak convergence of measures $\mu$ and the convergence of their absolute first moments $\int |y| \ud \mu \, y$, see \ref{remark: d agrees with 1-Wasserstein metric} and \ref{lemma: characterization of the metric topology}, respectively.

\textbf{Organization of this paper.}
In Section \ref{section: inductive limit and topological tensor product}, we collect the necessary materials for Section \ref{section: the test function spaces} from functional analysis. 
In Section \ref{section: Young functions and graph measures}, we define the notions of Young functions and their associated graph measures; also, we prove a compactness theorem for graph measures, a disintegration theorem, and therefore a compactness theorem for pairs of varifolds and Young functions.
In Section \ref{section: the space of probability Radon measures and operations on Young functions}, 
we study the space of probability Radon measures and define the operations pushforward, Cartesian product, and convolution on the class of Young functions. Moreover, we introduce Lipschitzian continuity for Young functions.
In Section~\ref{section: the test function spaces}, we study the aforementioned test function spaces $\mathbf{E}(Y)$ and $\mathbf{H}(U \times Y, \mathbf{R}^n)$, for which embedding theorems into spaces of continuous functions with compact support are established. 

\textbf{Terminology.} 
We employ the terms pseudo-metric, symmetric sets, absorbent sets, and locally convex spaces and locally convex topology, in accordance with Bourbaki, see \cite[IX, \S 1, No.\ 1, Definition 1]{Bourbaki_topology_II_MR979295}, \cite[I, \S 4, No.\ 1, p.\ 31]{Bourbaki_algebra_I_MR979982}, \cite[I, \S 1, No.\ 5, Definition 4]{Bourbaki_TVS_MR910295}, \cite[II, \S 4, No.\ 1, Definition 1]{Bourbaki_TVS_MR910295}, respectively. We term a function of class $k$ if and only if it is $k$-times continuously differentiable; moreover, we term a function of class $\infty$ if and only if it is of class $k$ whenever $k$ is a positive integer, see Federer \cite[3.1.11]{Federer69_MR0257325}.

\textbf{Notation.}
The set of positive integers is denoted by $\mathscr{P}$. The set $\mathscr{P}^\mathscr{P}$ consists of all functions $\mathscr{P} \to \mathscr{P}$. If $V$ is a vector space, $v \in V$, and $f \in \homomorphism(V, \mathbf{R})$, then we denote $\langle v, f \rangle = f(v)$. If $f$ is an $\overline{\mathbf{R}}$-valued function defined on a directed set $A$, then its limit and limit superior (see Kelley \cite[p.\ 66]{Kelley75_MR0370454}) are denoted by $\lim_{\alpha \in A} f(\alpha)$ and $\limsup_{\alpha \in A} f(\alpha)$, respectively. If $F$ is the inductive limit of locally convex spaces $F_\alpha$, see \ref{definition: inductive system and inductive limit}, then we denote $F = \varinjlim F_\alpha$. If $U$ is an open subset of a finite-dimensional Banach space and $Y$ is a Banach space, then $\mathscr{E}(U, Y)$ denotes the space of all functions of class $\infty$ from $U$ into $Y$ and $\mathscr{D}(U, Y) = \mathscr{E}(U, Y) \cap \{ f \with \text{$\spt f$ is compact} \}$. The notation of Allard (see \cite{Allard72_MR307015}) will be employed for the discussion of varifolds.

\textbf{Definitions in this paper.}
The notion of \textit{directed sets} is defined in \ref{definition: directed set}. 
The notions of \textit{inductive systems} and their \textit{inductive limits} are defined \ref{definition: inductive system and inductive limit}.
The notions of \textit{locally convex final topology} and \textit{projective tensor product} are defined in \ref{definition: locally convex final topology} and \ref{definition: projective tensor product}, respectively. 
The space $\mathscr{K}(X, Z)$ of continuous functions with compact support and the space $\mathbf{P}(X)$ of probability Radon measures as well as their topologies are defined in \ref{definition: space of continuous functions with compact support} and \ref{definition: space of probability Radon measures}, respectively. 
The notions of \textit{Young functions} and their associated \textit{graph measures} are defined in \ref{definition: Young function} and \ref{definition: graph measure}, respectively. 
The test function space $\mathbf{E}(Y)$ is introduced in \ref{definition: function space E}.
The pseudo-metric $d$ on $\mathbf{P}(X)$ is defined in \ref{definition: pseudo-metric d}.
The notion of \textit{Lipschitzian Young functions} is defined in \ref{definition: Lipschitz continuity Young function}.
The \textit{pushforward} and \textit{product} of Young functions are defined in \ref{definition: pushforward of Young function} and \ref{definition: product of Young functions}, respectively.
The \textit{convolutions of measures} and the \textit{convolution of Young functions} are defined in \ref{definition: convolution of measures} and \ref{definition: convolution of Young functions}, respectively.
The test function spaces $\widetilde{\mathbf{E}}(Y)$ and $\mathbf{H}(U \times Y, \mathbf{R}^n)$ are introduced in \ref{definition: function space widetilde E} and \ref{definition: function space H}, respectively.

%%%%%%%%%%%%%%%%%%%%%%%%%%%%%%%%%%%%%%%%%%%%%%%%%%%%%%%%%%%%%%%%%%%%%%%%%%%%%%%%%%%%%%%%%%%%%%%%%%%%%%%%%%%
\section{Inductive limit and topological tensor product}
\label{section: inductive limit and topological tensor product}
%%%%%%%%%%%%%%%%%%%%%%%%%%%%%%%%%%%%%%%%%%%%%%%%%%%%%%%%%%%%%%%%%%%%%%%%%%%%%%%%%%%%%%%%%%%%%%%%%%%%%%%%%%%

In this section, we study the interaction between inductive limit and topological tensor product in the category of locally convex spaces.

%----------------------------------------------------------------------------------------------------------

\begin{definition}[see \protect{\cite[p.\ 65]{Kelley75_MR0370454}}]
    \label{definition: directed set}
    We say a set $A$ is \textit{directed by the relation $\preceq$} if and only if $\preceq$ is a reflexive and transitive relation on $A$ such that every subset of $A$ with at most two elements has an upper bound.
\end{definition}

%----------------------------------------------------------------------------------------------------------

\begin{definition}[see \protect{\cite[III, \S 7, No.\ 5, p.\ 202]{Bourbaki_TheoryOfSet_MR2102219} and \cite[II, \S 4, No.\ 4, Example II]{Bourbaki_TVS_MR910295}}]
    \label{definition: inductive system and inductive limit}
    Suppose $A$ is a set directed by the relation $\preceq$. An \textit{inductive system of sets [resp.\ locally convex spaces] relative to $A$} consists of a family of sets [resp.\ locally convex spaces] $F_\alpha$ for $\alpha \in A$ and a family of maps [resp.\ continuous linear maps] $f_{\beta, \alpha}$ for $\alpha, \beta \in A$ with $\alpha \preceq \beta$ such that $f_{\alpha, \alpha} = \mathbf{1}_{F_\alpha}$ whenever $\alpha \in A$ and such that $f_{\gamma, \beta} \circ f_{\beta, \alpha} = f_{\gamma, \alpha}$ whenever $\alpha, \beta, \gamma \in A$ with $\alpha \preceq \beta \preceq \gamma$. In this case, the \textit{inductive limit} of the inductive system $(F_\alpha, f_{\beta, \alpha})$ consists of a set [resp.\ a locally convex space] $F$ and a family of maps [resp.\ continuous linear maps] $f_\alpha$ such that $f_\beta \circ f_{\beta, \alpha} = f_\alpha$ whenever $\alpha, \beta \in A$ with $\alpha \preceq \beta$, and satisfies the following universal property: whenever $G$ is a set [resp.\ a locally convex space] and $g_\alpha: F_\alpha \to G$ is a map [resp.\ a continuous linear map] such that $g_\beta \circ f_{\beta, \alpha} = g_\alpha$ whenever $\alpha, \beta \in A$ with $\alpha \preceq \beta$, there exists a unique map [resp.\ a unique continuous linear map] $g: F \to G$ such that $g \circ f_\alpha = g_\alpha$ whenever $\alpha \in A$. For convenience, we will say $F$ is the inductive limit of $F_\alpha$ for $\alpha \in A$ if either the maps $f_{\beta, \alpha}$ and $f_\alpha$ do not play a role in the statements or these maps are clear from the context.
\end{definition}

%----------------------------------------------------------------------------------------------------------

\begin{remark}[see \protect{\cite[III, \S 7, No.\ 5, p.\ 202]{Bourbaki_TheoryOfSet_MR2102219}}]
    \label{remark: inductive limit of sets}
    Suppose $(F_\alpha, f_{\beta, \alpha})$ forms an inductive system of sets relative to $A$. By convention, we identify $F_\alpha$ with its canonical image in the disjoint union $\bigsqcup_{\alpha \in A} F_\alpha$. Define $F$ to be $\bigsqcup_{\alpha \in A} F_\alpha$ modulo the equivalence relation $R$ such that $(x, y) \in R$ for some $x \in F_\alpha$ and $y \in F_\beta$ if and only if there exists $\gamma \in A$ such that $\alpha \preceq \gamma$, $\beta \preceq \gamma$, and $f_{\gamma, \alpha}(x) = f_{\gamma, \beta}(y)$. For $\alpha \in A$, let $f_\alpha$ be the composition of the canonical maps $F_\alpha \to \bigsqcup_{\alpha \in A} F_\alpha \to F$. Then, the set $F$ together with the maps $f_\alpha$ forms the inductive limit of the inductive system $(F_\alpha, f_{\beta, \alpha})$ of sets. 
\end{remark}

%----------------------------------------------------------------------------------------------------------

\begin{definition}[see \protect{\cite[II, \S 4, No.\ 4, Proposition 5]{Bourbaki_TVS_MR910295}}]
    \label{definition: locally convex final topology}
    Suppose $A$ is a set, $F$ is a vector space, $F_\alpha$ is a locally convex space for $\alpha \in A$, and $f_\alpha: F_\alpha \to F$ is a linear map for $\alpha \in A$. Then, the \textit{locally convex final topology} is the finest locally convex topology such that $f_\alpha$ is continuous whenever $\alpha \in A$. A fundamental system of neighborhoods of $0$ of $F$ is given by the family of all absorbent, convex, symmetric subsets $V$ of $F$ such that $f_\alpha^{-1}[V]$ is a neighborhood of $0$ in $F_\alpha$ whenever $\alpha \in A$.
\end{definition}

%----------------------------------------------------------------------------------------------------------

\begin{remark}[see \protect{\cite[II, \S 4, No.\ 4, Proposition 5(ii)]{Bourbaki_TVS_MR910295}}]
    \label{remark: universal property locally convex final topology}
    The locally convex final topology satisfies the following universal property: if $G$ is a locally convex space and $L: F \to G$ is a linear map, then $L$ is continuous if and only if $L \circ f_\alpha$ is continuous whenever $\alpha \in A$.
\end{remark}

%----------------------------------------------------------------------------------------------------------

\begin{remark}[see \protect{\cite[II, \S 4, No.\ 4, Example II]{Bourbaki_TVS_MR910295}}]
    \label{remark: inductive limit of locally convex spaces}
    Suppose $A$ is a directed nonempty set, $(F_\alpha, f_{\beta, \alpha})$ forms an inductive system of locally convex spaces (hence an inductive system of sets) relative to $A$, and $F$ and $f_\alpha$ are as in \ref{remark: inductive limit of sets}. Then, $F$ can be endowed with a canonical vector structure such that each $f_\alpha$ is linear, see \cite[II, \S 6, No.\ 2, p. 286]{Bourbaki_algebra_I_MR979982}; moreover, if we endow $F$ with the locally convex final topology induced by the maps $f_\alpha$, then $F$ together with $f_\alpha$ for $\alpha \in A$ forms the inductive limit of the inductive system $(F_\alpha, f_{\beta, \alpha})$ of locally convex spaces.
\end{remark}

%----------------------------------------------------------------------------------------------------------

\begin{remark}[see \protect{\cite[III, \S 7, No.\ 6, Corollary 1 and Proposition 7]{Bourbaki_TheoryOfSet_MR2102219}}]
    \label{remark: inductive limit of maps}
    Suppose $A$ is a nonempty set directed by the relation $\preceq$, $(F_\alpha, f_{\beta, \alpha})$ and $(G_\alpha, g_{\beta, \alpha})$ are two inductive systems of sets [resp.\ locally convex spaces] relative to $A$, $(F, f_\alpha)$ and $(G, g_\alpha)$ are their inductive limits, respectively, and the maps [resp.\ continuous linear maps] $h_\alpha: F_\alpha \to G_\alpha$ for $\alpha \in A$ satisfy $h_\beta \circ f_{\beta, \alpha} = g_{\beta, \alpha} \circ h_\alpha$ whenever $\alpha, \beta \in A$ with $\alpha \preceq \beta$. By the universal property of inductive limits, there exists a unique map [resp.\ a unique continuous linear map] $h: F \to G$ such that $h \circ f_\alpha = g_\alpha \circ h_\alpha$ whenever $\alpha \in A$; in this case, $h$ is called the inductive limit of $h_\alpha$ for $\alpha \in A$. Moreover, if each $h_\alpha$ is injective [resp.\ surjective], then $h$ is injective [resp.\ surjective].
\end{remark}

%----------------------------------------------------------------------------------------------------------

\begin{definition}[see \protect{\cite[2.9]{Menne16a_MR3528825} and \cite[II, \S 4, No.\ 6, p.\ 33]{Bourbaki_TVS_MR910295}}]
    We say an inductive limit of locally convex spaces is \textit{strict} if and only if the continuous linear maps in the corresponding inductive system are homeomorphic embeddings.
\end{definition}

%----------------------------------------------------------------------------------------------------------

\begin{remark}
    \label{remark: criterion strict inductive limit}
    In the applications in Section \ref{section: the test function spaces}, it will always be the case that $F_\alpha$ is a subspace of $F_\beta$ whenever $\alpha \preceq \beta$, $F = \bigcup_{\alpha \in A} F_\alpha$, and the maps $f_{\beta, \alpha}$ and $f_\alpha$ are inclusion maps. In this case, $F$ is the strict inductive limit of $F_\alpha$ if and only if $F$ is endowed with the locally convex final topology induced by the inclusion maps $f_\alpha$ and each $f_{\beta, \alpha}$ is a homeomorphic embedding.
\end{remark}

%----------------------------------------------------------------------------------------------------------

\begin{lemma}[see \protect{\cite[II, \S 4, No.\ 6, Proposition 9]{Bourbaki_TVS_MR910295} and \cite[III, \S 1, No.\ 4, Proposition 6]{Bourbaki_TVS_MR910295}}]
    \label{lemma: properties strict inductive limit}
    Suppose $F = \bigcup_{i=1}^\infty F_i$ is the strict inductive limit of a sequence of Banach spaces $F_i$ with $F_i \subset F_{i+1}$ for $i \in \mathscr{P}$. Then, the following two statements hold.
    \begin{enumerate}
        \item 
        \label{lemma: properties strict inductive limit1}
            $F$ is complete and Hausdorff, and the given topology on $F_i$ agrees with the subspace topology induced by $F$; in particular, $F_i$ is a closed subspace of $F$.
        \item 
        \label{lemma: properties strict inductive limit2}
            Every bounded subset (in particular, compact subset) of $F$ is contained in some $F_i$ and bounded there.
    \end{enumerate}
\end{lemma}

%----------------------------------------------------------------------------------------------------------

\begin{lemma}[see \protect{\cite[II, \S 4, No.\ 4, Corollary 2]{Bourbaki_TVS_MR910295}}]
    \label{lemma: double locally convex final topology}
    Suppose $A$ and $L$ are sets, $\{ J_\lambda \with \lambda \in L \}$ is a partition of $A$, $G_\alpha$ is a locally convex space for $\alpha \in A$, $F_\lambda$ is a vector space for $\lambda \in L$, $E$ is a vector space, $g_{\lambda, \alpha}: G_\alpha \to F_\lambda$ is a linear map for $\alpha \in J_\lambda$ and $\lambda \in L$, $h_\lambda : F_\lambda \to E$ is a linear map for $\lambda \in L$, and we endow $F_\lambda$ with the locally convex final topology with respect to $g_{\lambda, \alpha}$ for $\alpha \in J_\lambda$ and $\lambda \in L$.
    
    Then, the locally convex final topology on $E$ induced by $h_\lambda$ for $\lambda \in L$ agrees with the locally convex final topology induced by $h_\lambda \circ g_{\lambda, \alpha}$ for $\alpha \in J_\lambda$ and $\lambda \in L$.
\end{lemma}

%----------------------------------------------------------------------------------------------------------

\begin{definition}[see \protect{\cite[II, \S 4, No.\ 5, Definition 2]{Bourbaki_TVS_MR910295}}]
    Suppose $A$ is a set, $F_\alpha$ is a locally convex space for $\alpha \in A$, and $F$ is the direct sum of $F_\alpha$. Then, the \textit{direct sum topology on $F$} is the locally convex final topology induced by the canonical injections $F_\alpha \to F$.
\end{definition}

%----------------------------------------------------------------------------------------------------------

\begin{lemma}
    \label{lemma: inductive limit and direct sum}
    Suppose $A$ and $B$ are sets, $F_\alpha^\beta$ and $F_\alpha$ for $(\alpha, \beta) \in A \times B$ are locally convex spaces, $F_\alpha$ is endowed with the locally convex final topology induced by the linear maps $f_\alpha^\beta: F_\alpha^\beta \to F_\alpha$ for $\beta \in B$ whenever $\alpha \in A$. Then, the locally convex final topology $\mathcal{T}$ on $\bigoplus_{\alpha \in A} F_\alpha $ induced by the linear maps $\bigoplus_{\alpha \in A} f_\alpha^\beta$ for $\beta \in B$ agrees with the direct sum topology $\mathcal{T}'$ of $\bigoplus_{\alpha \in A} F_\alpha$.
\end{lemma}
\begin{proof}
    Denote the canonical injections by
    \[
        i_\alpha: F_\alpha \to \bigoplus_{a \in A} F_a, \quad i_\alpha^\beta: F_\alpha^\beta \to \bigoplus_{a \in A} F_a^\beta.
    \]
    In view of \ref{lemma: double locally convex final topology}, $\mathcal{T}$ is the locally convex final topology induced by the linear maps $(\bigoplus_{a \in A} f_a^\beta) \circ i_\alpha^\beta$ for $(\alpha, \beta) \in A \times B$, and $\mathcal{T}'$ is the locally convex final topology induced by the linear maps $i_\alpha \circ f_\alpha^\beta$ for $(\alpha, \beta) \in A \times B$. On the other hand, we have
    \[
        \left( \bigoplus_{a \in A} f_a^\beta \right) \circ i_\alpha^\beta = i_\alpha \circ f_\alpha^\beta
    \]
    whenever $(\alpha, \beta) \in A \times B$. It follows that $\mathcal{T} = \mathcal{T}'$.
\end{proof}

%----------------------------------------------------------------------------------------------------------

\begin{lemma}[see \protect{\cite[II, \S 4, No.\ 3, Proposition 4]{Bourbaki_TVS_MR910295}}]
    \label{lemma: Cartesian product of locally convex spaces}
    Suppose $A$ is a set, $F_\alpha$ is a locally convex space for $\alpha \in A$, $F$ is a vector space, and $f_\alpha: F \to F_ \alpha$ is a linear map for $\alpha \in A$. Then, the initial topology on $F$ with respect to $f_\alpha$ for $\alpha \in A$ is a locally convex topology; in particular, the Cartesian product of locally convex spaces endowed with the product topology is again a locally convex space.
\end{lemma}

%----------------------------------------------------------------------------------------------------------

\begin{remark}
    \label{remark: universal property product locally convex space}
    The following universal property of the Cartesian product of locally convex spaces follows from the corresponding property of topological spaces. Suppose $A$ is a set, $G$ and $F_\alpha$ for $\alpha \in A$ are locally convex spaces, $F$ is the Cartesian product of $F_\alpha$ for $\alpha \in A$, $f_\alpha$ is the projection map for each $\alpha \in A$, and $g_\alpha: G \to F_\alpha$ is a continuous linear map for each $\alpha \in A$. Then, there exists a unique continuous linear map $h: G \to F$ such that $f_\alpha \circ h = g_\alpha$ whenever $\alpha \in A$.
\end{remark}

%----------------------------------------------------------------------------------------------------------

\begin{lemma}[see \protect{\cite[II, \S 4, No.\ 5, Proposition 7]{Bourbaki_TVS_MR910295}}]
    \label{lemma: sum = product for finitely many copies}
    Suppose $A$ is a set, $F_\alpha$ is a locally convex space for $\alpha \in A$. Then, the canonical injection 
    \[
        \bigoplus_{\alpha \in A} F_\alpha \to \prod_{\alpha \in A} F_\alpha
    \]
    is continuous. If $A$ is finite, then this map is an isomorphism of locally convex spaces.
\end{lemma}

%----------------------------------------------------------------------------------------------------------

\begin{theorem}
    \label{theorem: inductive limit and finite product}
    Suppose $(F, f_\alpha)$ is the inductive limit [resp.\ the strict inductive limit] of the inductive system $(F_\alpha, f_{\beta, \alpha})$ of locally convex spaces. Then, $(F^n, \prod_{i = 1}^n f_\alpha)$ is the inductive limit [resp.\ the strict inductive limit] of the inductive system $((F_\alpha)^n, \prod_{i = 1}^n f_{\beta, \alpha})$ of locally convex spaces.
\end{theorem}
\begin{proof}
    By \ref{lemma: Cartesian product of locally convex spaces} and \ref{remark: universal property product locally convex space}, we readily check that $((F_\alpha)^n, \prod_{i = 1}^n f_{\beta, \alpha})$ form an inductive system of locally convex spaces. Observe that $(F^n, \prod_{i=1}^n f_\alpha)$ is the inductive limit of $((F_\alpha)^n, \prod_{i = 1}^n f_{\beta, \alpha})$ in the category of sets. In view of \ref{remark: inductive limit of locally convex spaces}, it is enough to show that $F^n$ carries the locally convex final topology induced by the maps $\prod_{i = 1}^n f_\alpha$ which is immediate by \ref{lemma: sum = product for finitely many copies} and \ref{lemma: inductive limit and direct sum}. To show the strictness of the inductive limit, we note that the Cartesian product of homeomorphic embeddings is again a homeomorphic embedding.
\end{proof}

%----------------------------------------------------------------------------------------------------------

\begin{definition}[see \protect{\cite[III, 6.1 -- 6.3]{SW99_MR1741419}}]
    \label{definition: projective tensor product}
    Suppose $V$ and $W$ are locally convex spaces. The \textit{projective tensor product topology} of $V \otimes W$ is the finest locally convex topology on $V \otimes W$ such that the natural bilinear map $\mu: V \times W \to V \otimes W$ is continuous.
    
    Furthermore, if $V$ and $W$ are normed, then $V \otimes W$ is normed by the \textit{projective tensor product norm} with value
    \[
        \inf \left\{ \sum_{i=1}^n |v_i||w_i| \with \xi = \sum_{i=1}^n v_i \otimes w_i, v_i \in V, w_i \in W, i = 1, \dotsc, n \right\}
    \]
    at $\xi \in V \otimes W$.
\end{definition}

%----------------------------------------------------------------------------------------------------------

\begin{remark}
    \label{remark: projective tensor product universal property}
    The projective tensor product topology satisfies the following property: whenever $Z$ is a locally convex space and $g: V \times W \to Z$ is a continuous bilinear map, there exists a unique continuous linear map $h: V \otimes W \to Z$ such that $h = g \circ \mu$. In case that $V$, $W$ and $Z$ are normed, we have $\|g\| = \|h\|$.
\end{remark}

%----------------------------------------------------------------------------------------------------------

\begin{remark}
    \label{remark: tensor product continuous functions}
    Suppose $f_1: V_1 \to W_1$ and $f_2: V_2 \to W_2$ are continuous linear maps between locally convex spaces. Then, their tensor product $f_1 \otimes f_2: V_1 \otimes V_2 \to W_1 \otimes W_2$ is continuous with respect to the projective tensor product topologies.
\end{remark}

%----------------------------------------------------------------------------------------------------------

\begin{remark}
    \label{remark: tensor isomorphism}
    Suppose $V$, $W$, and $Z$ are normed spaces. It is straightforward to verify the following isometries of normed spaces
    \[
        (V \otimes W) \otimes Z \simeq V \otimes (W \otimes Z), \quad V \otimes W \simeq W \otimes V.
    \]
    If $\dimension V < \infty$, we have the following isomorphism of normed spaces
    \[
        \homomorphism(V, \mathbf{R}) \otimes W \simeq \homomorphism(V, W);
    \]
    however, in general this is not an isometry.
\end{remark}

%----------------------------------------------------------------------------------------------------------

\begin{remark}
    \label{remark: homeomorphism inductive tensor product}
    Associating with a locally convex space $V$, the following isomorphism of locally convex spaces
    \[
        I_V: \mathbf{R}^n \otimes V \simeq V^n.
    \]
    yields a natural transformation; that is, whenever $W$ is locally convex space and $f: V \to W$ is a continuous linear map, we have $I_W \circ (\mathbf{1}_{\mathbf{R}^n} \otimes f) = (\prod_{i = 1}^n f) \circ I_V$.
\end{remark}

%----------------------------------------------------------------------------------------------------------

\begin{theorem}
    \label{theorem: tensor product of inductive limit}
    Suppose $(F, f_\alpha)$ is the inductive limit [resp.\ the strict inductive limit] of the inductive system $(F_\alpha, f_{\beta, \alpha})$ of locally convex spaces. Then, $(\mathbf{R}^n \otimes F, \mathbf{1}_{\mathbf{R}^n} \otimes f_\alpha)$ is the inductive limit [resp.\ the strict inductive limit] of the inductive system $(\mathbf{R}^n \otimes F_\alpha, \mathbf{1}_{\mathbf{R}^n} \otimes f_{\beta, \alpha})$ of locally convex spaces.
\end{theorem}
\begin{proof}
    Combine \ref{theorem: inductive limit and finite product} and \ref{remark: homeomorphism inductive tensor product}.
\end{proof}

%%%%%%%%%%%%%%%%%%%%%%%%%%%%%%%%%%%%%%%%%%%%%%%%%%%%%%%%%%%%%%%%%%%%%%%%%%%%%%%%%%%%%%%%%%%%%%%%%%%%%%%%%%%
\section{Young functions and graph measures}
\label{section: Young functions and graph measures}
%%%%%%%%%%%%%%%%%%%%%%%%%%%%%%%%%%%%%%%%%%%%%%%%%%%%%%%%%%%%%%%%%%%%%%%%%%%%%%%%%%%%%%%%%%%%%%%%%%%%%%%%%%%

In this section, we introduce Young functions and graph measures, and we will focus on the compactness property of graph measures, leaving the discussion about Young functions to the next section.

We first collect necessary materials about the test function space $\mathscr{K}(X, Z)$ for later discussions on Young functions and graph measures.

\begin{definition}[see \protect{\cite[III, \S 1, No.\ 1]{Bourbaki_integration_MR2018901}}]
    \label{definition: space of continuous functions with compact support}
    Suppose $X$ is a locally compact Hausdorff space and $Z$ is a locally convex space. Then, the space $\mathscr{K}_K(X, Z)$ of continuous functions from $X$ into $Z$ with compact support in $K$ is endowed with the supremum norm and the space
    \[
        \mathscr{K}(X, Z) =  \bigcup \{ \mathscr{K}_K(X, Z) \with \text{$K$ is a compact subset of $X$} \}
    \]
    is endowed with the locally convex final topology with respect to the inclusion maps $\mathscr{K}_K(X, Z) \to \mathscr{K}(X, Z)$ for compact subsets $K$ of $X$. In case $Z = \mathbf{R}$, we will simply write $\mathscr{K}_K(X)$ and $\mathscr{K}(X)$; also, we denote $\mathscr{K}(X)^+ = \mathscr{K}(X) \cap \{ f \with f \geq 0 \}$.
\end{definition}

%----------------------------------------------------------------------------------------------------------

\begin{remark}
    The topological dual space $\mathscr{K}(X)^*$ equals the space of all Daniell integrals on $\mathscr{K}(X)$, see \cite[2.5.13]{Federer69_MR0257325}, and it is endowed with the weak topology; that is, the initial topology induced from the maps $\mu \mapsto \mu(f)$ corresponding to $f \in \mathscr{K}(X)$, see \cite[2.5.19]{Federer69_MR0257325}. For $\mu \in \mathscr{K}(X)^*$ such that $\mu(f) \geq 0$ whenever $f \in \mathscr{K}(X)^+$, there exists a unique Radon measure over $X$ representing $\mu$, see \cite[2.5.13, 2.5.14]{Federer69_MR0257325}.
\end{remark}

%----------------------------------------------------------------------------------------------------------

\begin{lemma}
    \label{lemma: exponential law}
    Suppose $X$ and $Y$ are locally compact Hausdorff spaces, $K$ and $L$ are compact subsets of $X$ and $Y$, respectively, $Z$ is a Banach space. and the map
    \[
        \iota: Z^{X \times Y} \to (Z^Y)^X
    \]
    is characterized by
    \[
        \iota(\eta)(x) = \eta(x, \cdot) \quad \text{whenever $x \in X$ and $\eta \in Z^{X \times Y}$},
    \]
    where $B^A$ is the set of all functions $f: A \to B$. Then, the restriction 
    \[
        \iota | \mathscr{K}_{K \times L}(X \times Y, Z): \mathscr{K}_{K \times L}(X \times Y, Z) \to \mathscr{K}_K(X, \mathscr{K}_L(Y, Z))
    \]
    defines a norm-preserving isomorphism of Banach spaces.
\end{lemma}
\begin{proof}
    Clearly, we have $\sup \im |\eta| = \sup \im |\iota(\eta)|$ whenever $\eta \in \mathscr{K}_{K \times L}(X \times Y, Z)$. For topological spaces $A$ and $B$, we denote by $\mathscr{C}(A, B)$ the space of all continuous functions $f: A \to B$ endowed with the compact-open topology. In case that $B$ is a normed space, it is straightforward to verify that this topology agrees with the locally convex topology induced by the semi-norms with value $\sup |f|[K]$ at $f \in \mathscr{C}(A, B)$ corresponding to compact subsets $K$ of $A$. It follows that the inclusion map $\mathscr{K}_K(X, B) \to \mathscr{C}(X, B)$ is a homeomorphic embedding whenever $B$ is a normed space and $K$ is a compact subset of $X$. Since the restriction 
    \[
        \iota | \mathscr{C}(X \times Y, Z): \mathscr{C}(X \times Y, Z) \to \mathscr{C}(X, \mathscr{C}(Y, Z))
    \]
    defines a bijection by \cite[2.4.7]{tD08_MR2456045}, the assertion is obvious.
\end{proof}

%----------------------------------------------------------------------------------------------------------

\begin{remark}
    The above lemma is a generalization of \cite[III, \S 4, No.\ 1, Lemma 1(i)]{Bourbaki_integration_MR2018901} to vector-valued functions.
\end{remark}

%----------------------------------------------------------------------------------------------------------

\begin{lemma}[see \protect{\cite[III, \S 1, No.\ 2, Proposition 5]{Bourbaki_integration_MR2018901}}]
    \label{lemma: K(X) tensor Z density}
    Suppose $X$ is a locally compact Hausdorff space, $Z$ is a locally convex space, and $K$ is a compact subset of $X$. Then, the image of the canonical map $\mathscr{K}_K(X) \otimes Z \to \mathscr{K}_K(X, Z)$ is dense.
\end{lemma}

%----------------------------------------------------------------------------------------------------------

\begin{lemma}[see \protect{\cite[III, \S 4, No.\ 1, Lemma 1(ii)]{Bourbaki_integration_MR2018901}}]
    \label{lemma: K(X) tensor K(Y) density}
    Suppose $X$ and $Y$ are locally compact Hausdorff spaces, $K$ and $L$ are compact subsets of $X$ and $Y$, respectively. Then, the image of the canonical map $\mathscr{K}_K(X) \otimes \mathscr{K}_L(Y) \to \mathscr{K}_{K \times L}(X \times Y)$ is dense.
\end{lemma}
\begin{proof}
    Apply \ref{lemma: K(X) tensor Z density} with $Z = \mathscr{K}_L(Y)$ and the assertion follows from \ref{lemma: exponential law}.
\end{proof}

%----------------------------------------------------------------------------------------------------------

\begin{definition}
    \label{definition: space of probability Radon measures}
    Suppose $X$ is a locally compact Hausdorff space. We define $\mathbf{P}(X)$ to be the space of all probability Radon measures over $X$, and it is endowed with the subspace topology induced from $\mathscr{K}(X)^*$.
\end{definition}

%----------------------------------------------------------------------------------------------------------

\begin{lemma}
    \label{lemma: two weak topologies agree}
    Suppose $X$ is a locally compact Hausdorff space. Then, the weak topology on $\mathbf{P}(X)$ agrees with the initial topology induced from the maps $\mu \mapsto \textstyle\int g \ud \mu$ corresponding to bounded continuous functions $g: X \to \mathbf{R}$.

    Consequently, a sequence $\mu_i$ converges to $\mu$ weakly if and only if $\textstyle\int g \ud \mu_i \to \textstyle\int g \ud \mu$ whenever $g: X \to \mathbf{R}$ is a bounded continuous function.
\end{lemma}
\begin{proof}
    Clearly, the weak topology of $\mathbf{P}(X)$ is weaker than the initial topology. Denoting by $e_g(\mu) = \textstyle\int g \ud \mu$, then it is enough to show that $e_g: \mathbf{P}(X) \to \mathbf{R}$ is continuous with respect to the weak topology whenever $g: X \to \mathbf{R}$ is a non-negative bounded continuous function. 
    By \cite[2.2.5]{Federer69_MR0257325} and approximation with functions in $\mathscr{K}(X)$, we can show that
    \[
        \mu(g) = \sup \{ \mu(f) \with 0 \leq f \leq g, f \in \mathscr{K}(X) \}
    \]
    whenever $\mu \in \mathbf{P}(X)$ and $g: X \to \mathbf{R}$ is a non-negative bounded continuous function; in this case, it follows that $e_g$ equals the supremum of the family of continuous functions $e_f$ corresponding to $f \in \mathscr{K}(X)$ with $0 \leq f \leq g$, and therefore $e_g$ is lower semi-continuous. Replacing $g$ with $M - g$, where $M$ is a large number such that $M - g \geq 0$, we have $e_g = -(e_{M-g} - M)$ is upper semi-continuous.
\end{proof}

%----------------------------------------------------------------------------------------------------------

\begin{remark}
    The initial topology mentioned in the lemma is usually termed weak topology in texts on probability theory (see \cite[13.14(ii)]{Kle20_MR4201399}), and the lemma shows that these two notions of weak topology agree.
\end{remark}

%----------------------------------------------------------------------------------------------------------

\begin{definition}
    \label{definition: Young function}
    Suppose $X$ and $Y$ are locally compact Hausdorff spaces and $\mu$ is a Radon measure over $X$. By a \textit{$\mu$ Young function $f$ of type $Y$}, we mean a $\mu$ measurable function with values in $\mathbf{P}(Y)$.
\end{definition}

%----------------------------------------------------------------------------------------------------------

\begin{remark}
    \label{remark: single-valued Young function}
    Suppose $\boldsymbol{\updelta}: Y \to \mathscr{K}(Y)^*$ is given by $\boldsymbol{\updelta}_y (\beta) = \beta(y)$ whenever $y \in Y$ and $\beta \in \mathscr{K}(Y)$, $g$ maps a subset of $X$ into $Y$, $\mu(X \without \domain g) = 0$, and $f = \boldsymbol{\updelta} \circ g$. Recall from \cite[2.5.19]{Federer69_MR0257325} that $\boldsymbol{\updelta}$ is a homeomorphic embedding, it follows that $g$ is $\mu$ measurable if and only if $f$ is $\mu$ measurable.
\end{remark}

%----------------------------------------------------------------------------------------------------------

\begin{remark}
    \label{remark: measurability criterion}
    In view of \cite[2.12, 2.22]{Menne16a_MR3528825}, if $Y$ is second-countable, then a $\mathbf{P}(Y)$-valued function $f$ defined for $\mu$ almost everywhere is $\mu$ measaurable if and only if $x \mapsto \textstyle\int \beta \ud f(x)$ is $\mu$ measurable whenever $\beta \in \mathscr{K}(Y)$.
\end{remark}

%----------------------------------------------------------------------------------------------------------

\begin{example}
    \label{example: Almgren's Q-valued function}
    Let $m, Q \in \mathscr{P}$. Suppose $f$ is a $\mathscr{L}^m$ measurable $\mathbf{Q}_Q(\mathbf{R}^n)$-valued function in the sense of \cite[1.1]{Almgren2000_MR1777737}. By \cite[0.4]{DLS11_MR2663735}, there exists $\mathscr{L}^m$ measurable $\mathbf{R}^n$-valued functions $f_1, f_2, \dotsc, f_Q$ such that $f= \boldsymbol{\updelta} \circ f_1 + \boldsymbol{\updelta} \circ f_2 + \dotsm + \boldsymbol{\updelta} \circ f_Q$; in particular, $Q^{-1}f$ is a $Q \mathscr{L}^m$ Young function of type $\mathbf{R}^n$ by \ref{remark: measurability criterion}.
\end{example}

%----------------------------------------------------------------------------------------------------------

\begin{definition}
    \label{definition: graph measure}
    Suppose $X$ and $Y$ are locally compact Hausdorff spaces, $\mu$ is a Radon measure over $X$, and $f$ is $\mu$ Young function of type $Y$. Then the \textit{graph measure $\mathbf{Y}(\mu, f)$ associated with $\mu$ and $f$} is the Radon measure over $X \times Y$ such that (see \ref{remark: graph measures well-defined})
    \[
        \mathbf{Y}(\mu, f)(\psi) = \textstyle\iint \psi(x, y) \ud f(x) \, y \ud \mu \, x 
    \]
    whenever $\psi \in \mathscr{K}(X \times Y)$.
\end{definition}

%----------------------------------------------------------------------------------------------------------

\begin{remark}
    \label{remark: graph measures well-defined}
    The notion of graph measure is well-defined. By \ref{lemma: K(X) tensor K(Y) density}, for $\psi \in \mathscr{K}(X \times Y)$ and $i \in \mathscr{P}$, there exist $k_i \in \mathscr{P}$, $\alpha_j \in \mathscr{K}(X)$ and $\beta_j \in \mathscr{K}(Y)$ for $j = 1, \dotsc, k_i$ such that 
    \[ 
        \left| \psi(x, y) - \sum_{j = 1}^{k_i} \alpha_j(x) \beta_j(y) \right| \leq i^{-1} \quad \text{whenever $(x, y) \in X \times Y$}
    \]
    and hence
    \[ 
        \lim_{i \to \infty} \left| \textstyle\int \psi(x, y) \ud f(x) \, y - \sum\limits_{j = 1}^{k_i} \alpha_j(x)\textstyle\int \beta_j(y) \ud f(x) \, y \right| = 0 \quad \text{for $\mu$ almost all $x$.} 
    \]
    Thus, the map $x \mapsto \textstyle\int \psi(x, y) \ud f(x) \, y$ is $\mu$ measurable.
\end{remark}

%----------------------------------------------------------------------------------------------------------

\begin{remark}
    \label{remark: comparison measure-function pair}
    Suppose $m, n \in \mathscr{P}$ and $g$ is a $\mathscr{L}^m$ measurable function with values in $\mathbf{R}^n$. Then, $f = \boldsymbol{\updelta} \circ g$ is a $\mathscr{L}^m$ Young function of type $\mathbf{R}^n$; 
    letting $G(x) = (x, g(x))$ for $x \in \domain g$, we have that $G_\# \mathscr{L}^m$ is a Radon measure by \cite[2.11]{MS25_MR4864996} and \cite[2.2.2]{Federer69_MR0257325}, and that the graph measure $\mathbf{Y}(\mathscr{L}^m, f)$ associated with the pair $(\mathscr{L}^m, f)$ equals $G_\# \mathscr{L}^m$. Compared with the notion of measure-function pairs introduced by Hutchinson (see \cite[4.1.1]{Hutchinson86_MR825628}), our setting does not require any summability condition on functions $g$; also, this example shows our notion of graph measure generalizes the one in \cite[4.3.1]{Hutchinson86_MR825628}.
\end{remark}

%----------------------------------------------------------------------------------------------------------

\begin{remark}
    \label{remark: Young measures}
    Young functions are named after Laurence C. Young for their connection with Young measures. Suppose $m, n \in \mathscr{P}$ and $U$ is a bounded open subset of $\mathbf{R}^m$. Using our terminology, the Young measures associated with a bounded sequence of functions $g_i$ in $\mathbf{L}_{\infty}(\mathscr{L}^m \restrict U)$ form a $\mathscr{L}^m \restrict U$ Young function $f$ of type $\mathbf{R}^n$ such that 
    \[
        \mathbf{Y}(\mathscr{L}^m \restrict U, f) = \lim_{i \to \infty} \mathbf{Y}(\mathscr{L}^m \restrict U, \boldsymbol{\updelta} \circ g_i),
    \]
    see \cite[2.30(i)]{AFP00_MR1857292} and \ref{remark: graph measures well-defined}.
\end{remark}

%----------------------------------------------------------------------------------------------------------

\begin{example}
    \label{example: varifold disintegration}
    Suppose $m, n \in \mathscr{P}$ and $U$ is an open subset of $\mathbf{R}^n$. Employing the notation of Allard's varifold disintegration (see \cite[3.3]{Allard72_MR307015}), for each varifold $V \in \mathbf{V}_m(U)$, the formula $x \mapsto V^{(x)}$ defines a $\|V\|$ Young function of type $\mathbf{G}(n, m)$ such that $V = \mathbf{Y}(\|V\|, V^{(\cdot)})$. A similar statement also holds for Young's generalized surfaces (which are termed oriented varifolds in \cite[Section 3]{Hutchinson86_MR825628}), see \cite[Part I, Section 4, (b), Section 10]{Fle20_MR4077153} and \cite[Section 3]{Fle57_MR95923}.
\end{example}

%----------------------------------------------------------------------------------------------------------

The following lemma shows that the disintegration formula in the definition of graph measures holds for a much larger class of functions.

%----------------------------------------------------------------------------------------------------------

\begin{lemma}
    \label{lemma: extension of disintegration formula to integrable functions}
    Suppose $X$ and $Y$ are second-countable locally compact Hausdorff spaces, $\mu$ is a Radon measure over $X$, and $f$ is a $\mu$ Young function of type $Y$. Then, we have
    \[
        \textstyle\int \psi \ud \mathbf{Y}(\mu, f) = \textstyle\iint \psi(x, y) \ud f(x) \, y \ud \mu \, x 
    \]
    whenever $\psi$ is a $\mathbf{Y}(\mu, f)$ integrable $\overline{\mathbf{R}}$-valued function.
\end{lemma}
\begin{proof}
    Let $F$ consist of all $\mathbf{Y}(\mu, f)$ measurable subsets $A$ of $X \times Y$ such that the formula remains true with $\psi$ replaced by the characteristic function $\chi_A$ of $A$. 
    
    \textbf{Step 1.}
    Note that every open subset of $X \times Y$ is a countable union of compact subsets of $X \times Y$, and the disintegration formula holds for $\psi \in \mathscr{K}(X \times Y)$. In view of \cite[2.4.7]{Federer69_MR0257325}, $F$ contains all the open subsets of $X \times Y$ and $F$ is stable under countable increasing unions. If $A_i \in F$ for $i \in \mathscr{P}$ satisfies $\mathbf{Y}(\mu, f)(A_1) < \infty$ and $A_{i+1} \subset A_i$ for $i \in \mathscr{P}$, then $\bigcap_{i=1}^\infty A_i \in F$ by \cite[2.4.9]{Federer69_MR0257325}.
    
    \textbf{Step 2.}
    If $\mathbf{Y}(\mu, f)(A) = 0$, there exists a countable nonempty subfamily $G \subset F$ of open subsets in $X \times Y$ such that $A \subset \bigcap G$ and $\mathbf{Y}(\mu, f)(\bigcap G) = 0$. Therefore, we have
    \[ 
        \textstyle\int^* \chi_A(x, \cdot) \ud f(x) \leq \textstyle\int \chi_{\bigcap G}(x, \cdot) \ud f(x) = 0 \quad \text{for $\mu$ almost all $x$} 
    \]
    and this shows $F$ contains all the sets $A$ with $\mathbf{Y}(\mu, f)(A) = 0$.
    
    \textbf{Step 3.}
    If $A$ is a $\mathbf{Y}(\mu, f)$ measurable set with $\mathbf{Y}(\mu, f)(A) < \infty$, then there exists a countable nonempty subfamily $G \subset F$ of open subsets in $X \times Y$ such that $A \subset \bigcap G$ and $\mathbf{Y}(\mu, f)((\bigcap G) \without A) = 0$, hence
    \[
        \textstyle\int \chi_{(\bigcap G) \without A}(x, \cdot) \ud f(x) = 0 \quad \text{for $\mu$ almost all $x$};
    \]
    it follows that
    \[ 
        \textstyle\int \chi_A(x, \cdot) \ud f(x) = \textstyle\int \chi_{\bigcap G}(x, \cdot) \ud f(x) \quad \text{for $\mu$ almost all $x$},
    \]
    hence $A \in F$. Note that $X \times Y$ is countably $\mathbf{Y}(\mu, f)$ measurable. Therefore, $F$ contains all $\mathbf{Y}(\mu, f)$ measurable sets.
    
    Now, the assertion follows from \cite[2.3.3, 2.4.8, 2.4.4(6)]{Federer69_MR0257325}.
\end{proof}

%----------------------------------------------------------------------------------------------------------

To prove the compactness theorem and disintegration theorem, we want to reduce the problem by replacing $Y$ with its one-point compactification $Z$; for this purpose, we study the relation between the Radon measures over $Y$ and their image under the inclusion map $\iota: Y \to Z$.

%----------------------------------------------------------------------------------------------------------

\begin{lemma}
    \label{lemma: compactness implies uniform smallness at infinity}
    Suppose $Z$ is a compact Hausdorff space, $C$ is a closed subset of $Z$, and $M$ is a nonempty compact family of Radon measures over $Z$ such that $\phi(C) = 0$ for all $\phi \in M$. Then, there holds
    \[
        \lim_{L \in P} \sup \{ \phi(Z \without L) \with \phi \in M \} = 0,
    \]
    where $P$ is the family of compact subsets of $Z \without C$ directed by the order $\preceq$ such that $L_1 \preceq L_2$ if and only if $L_1 \subset L_2$ for $L_1, L_2 \in P$.
\end{lemma}
\begin{proof}
    Let $\lambda = \limsup_{L \in P} \sup \{ \phi(Z \without L) \with \phi \in M \}$ and let $F$ consist of those $f \in \mathscr{K}(Z)$ satisfying $0 \leq f \leq 1$, and $C \subset \interior \{ z \with f(z) = 1 \}$. Since $M$ is compact, whenever $f \in F$, there exists $\psi_f \in M$ such that
    \[
        \psi_f(f) \geq \phi(f) \quad \text{whenever $\phi \in M$};
    \]
    furthermore, letting $L = Z \without \interior\{ z \with f(z) = 1 \}$, we have $L \in P$ and $\psi_f(f) \geq \phi(Z \without L)$ whenever $\phi \in M$, hence $\psi_f(f) \geq \lambda$. Observe that $\{ \closure \{ \psi_g \with f \geq g \in F \} \with f \in F\}$ forms a family of closed subsets of $M$ that possesses the finite intersection property, then there exists
    \[
        \phi \in \bigcap_{f \in F} \closure \{ \psi_g \with f \geq g \in F \}.
    \]
    Since $\psi_g(f) \geq \psi_g(g) \geq \lambda$ and $f \mapsto \psi(f)$ for $\psi \in M$ is continuous whenever $f \geq g \in F$, we conclude $\phi(f) \geq \lambda$ whenever $f \in F$, thus $\lambda \leq \inf \{ \phi(f) \with f \in F \} = \phi(C) = 0$.
\end{proof}

%----------------------------------------------------------------------------------------------------------

\begin{lemma}
    \label{lemma: compactness equivalence}
    Suppose $X$ and $Y$ are locally compact Hausdorff spaces, $Z$ is the one-point compactification of $Y$, and $M$ is a family of Radon measures over $X \times Z$. Then, the following statements are equivalent.
    \begin{enumerate}
        \item 
            \label{lemma: compactness equivalence 1}
            $\closure M$ is compact and $\phi(K \times (Z \without Y)) = 0$ whenever $K$ is a compact subset of $X$ and $\phi \in \closure M$.
        \item 
            \label{lemma: compactness equivalence 2}
            $\sup \{ \phi(K \times Z) \with \phi \in M \} < \infty$ whenever $K$ is a compact subset of $X$ and
            \[
                \lim_{L \in P} \sup \{ \phi(K \times (Z \without L)) \with \phi \in M \} = 0,
            \]
            where $P$ is the family of compact subsets of $Y$ directed by the order $\preceq$ such that $L_1 \preceq L_2$ if and only if $L_1 \subset L_2$ for $L_1, L_2 \in P$.
    \end{enumerate}
\end{lemma}
\begin{proof}
    To show \eqref{lemma: compactness equivalence 1} implies \eqref{lemma: compactness equivalence 2}, we observe that for a compact subset $K$ of $X$ and a compact neighborhood $G$ of $K$, there exists $f \in \mathscr{K}_{G \times Z}(X \times Z)^+$ such that $0 \leq f \leq 1$ and $K \times Z \subset \{ x \with f(x) = 1 \}$, then mapping $\phi \in \closure M$ onto the finite Radon measure $(\phi \restrict f) | \mathbf{2}^{G \times Z}$ over $G \times Z$ defines a continuous map; from \ref{lemma: compactness implies uniform smallness at infinity}, we infer 
    \[
        \lim_{L \in P} \sup \{ \phi(K \times (Z \without L)) \with \phi \in M \} = 0.
    \]
    
    Conversely, for a compact subset $K$ of $X$, a compact neighborhood $G$ of $K$, and $L \in P$, there exists $f_L \in \mathscr{K}(X \times Z)$ such that 
    \[
        0 \leq f_L \leq 1, \quad K \times (Z \without Y) \subset \interior \{ x \with f_L(x) = 1 \}, \quad \spt f_L \subset G \times (Z \without L);
    \]
    it follows that
    \[
        \psi(K \times (Z \without Y)) \leq \psi(f_L) \leq \sup \{ \phi(G \times (Z \without L)) \with \phi \in M \},
    \]
    hence $\psi(K \times (Z \without Y)) = 0$ whenever $\psi \in \closure M$.
\end{proof}

%----------------------------------------------------------------------------------------------------------

\begin{lemma}
    \label{lemma: embedding of compact family of measures}
    Suppose $X$ and $Y$ are locally compact Hausdorff spaces, $Z$ is the one-point compactification of $Y$, $\iota: X \times Y \to X \times Z$ is the inclusion map, and the continuous monomorphism $j: \mathscr{K}(X \times Y) \to \mathscr{K}(X \times Z)$ is given by extension by zero to $X \times Z$. Then, there hold the following statements.
    \begin{enumerate}
        \item 
            \label{lemma: embedding of compact family of measures 1}
            If $X$ is a countable union of compact subsets and $\psi$ is a Radon measure over $X \times Y$ such that $\psi(K \times Y) < \infty$ whenever $K$ is a compact subset of $X$, then $\iota_\# \psi$ is a Radon measure over $X \times Z$.
        \item 
            \label{lemma: embedding of compact family of measures 2}
            If $\phi$ is a Radon measure over $X \times Z$ such that $\phi(K \times (Z \without Y)) = 0$ whenever $K$ is a compact subset of $X$ and $\psi$ is the Radon measure representing $j^*(\phi)$, then we have $\psi = \phi | \mathbf{2}^{X \times Y}$ and $\phi$ is the Radon measure representing the member in $\mathscr{K}(X \times Z)^*$ induced by $\iota_\# \psi$. 
        \item 
            \label{lemma: embedding of compact family of measures 3}
             Suppose $M$ is a compact family of Radon measures $\phi$ over $X \times Z$ such that $\phi(K \times (Z \without Y)) = 0$ whenever $K$ is a compact subset of $X$. The dual map $j^*: \mathscr{K}(X \times Z)^* \to \mathscr{K}(X \times Y)^*$ embeds $M$ homeomorphically into $\mathscr{K}(X \times Y)^*$.
    \end{enumerate}
\end{lemma}
\begin{proof}
    To prove \eqref{lemma: embedding of compact family of measures 1}, we first infer from \cite[2.1.2]{Federer69_MR0257325} that every open subset of $X \times Z$ is $\iota_\# \psi$ measurable. Next, noting that $X \times Y$ is an open subset of $X \times Z$, we have 
    \begin{align*}
        (\iota_\# \psi)(W) 
        &= \psi(W \cap (X \times Y)) \\
        &= \sup \{ \psi(C) \with \text{$C$ is compact, $C \subset W \cap (X \times Y)$} \} \\
        &\leq \sup \{ (\iota_\# \psi)(C) \with \text{$C$ is compact, $C \subset W$} \}
    \end{align*}
    whenever $W$ is an open subset of $X \times Z$. Let $\varepsilon > 0$ and let $K_i$ be a sequence of compact subsets of $X$ such that $X = \bigcup_{i=1}^\infty K_i$ and $K_{i} \subset \interior K_{i+1}$ for $i \in \mathscr{P}$. Since $\psi$ is a Radon measure and $\psi(K_i \times Y) < \infty$, there exists a sequence of compact subsets $L_i$ of $Y$ such that
    \[
        \psi(K_i \times (Y \without L_i)) < 2^{-i} \varepsilon.
    \]
    Then $W = \bigcup_{i=1}^\infty (\interior K_i) \times (Z \without L_i)$ is an open superset of $X \times (Z \without Y)$ such that $(\iota_\# \psi)(W) < \varepsilon$. For $A \subset X \times Z$, writing $A = (A \cap (X \times Y)) \cup (A \cap (X \times (Z \without Y)))$, from the estimate above and the fact that $\psi$ is a Radon measure, we conclude 
    \[
        (\iota_\# \psi)(A) = \inf \{ (\iota_\# \psi)(W) \with \text{$W$ is open, $A \subset W$} \}
    \]
    whenever $A \subset X \times Z$.

    As $j$ is continuous, so is $j^*$. To prove \eqref{lemma: embedding of compact family of measures 2}, we compute by \cite[2.4.18]{Federer69_MR0257325} that for $f \in \mathscr{K}(X \times Z)$ with $\spt f \subset X \times Y$,
    \[
        \textstyle\int f \ud (\iota_\# \psi) = \textstyle\int f \circ \iota \ud \psi = j^*(\phi)(f \circ \iota) = \phi(j(f \circ \iota)) = \textstyle\int f \ud \phi.
    \]
    Then, the first conclusion of \eqref{lemma: embedding of compact family of measures 2} follows. For arbitrary $f \in \mathscr{K}(X \times Z)$, letting $K$ be the image of projection of $\spt f$ onto $X$, there exists a compact subset $L$ of $Y$ such that $\phi(K \times (Z \without L))$ and $(\iota_\# \psi)(K \times (Z \without L))$ are small;  therefore, it allows us to approximate the integrals of $f$ with the integrals of those functions whose supports are contained in $X \times Y$, and the second conclusion of \eqref{lemma: embedding of compact family of measures 2} follows. By \eqref{lemma: embedding of compact family of measures 2}, $j^*|M$ is injective. Recalling that every continuous bijection from a compact space into a Hausdorff space is a homeomorphism, \eqref{lemma: embedding of compact family of measures 3} follows.
\end{proof}

%----------------------------------------------------------------------------------------------------------

\begin{theorem}[Compactness]
    \label{theorem: compactness of graph measures}
    Suppose $X$ and $Y$ are locally compact Hausdorff spaces, $X$ is a countable union of compact subsets, and $M$ is a class of Radon measures over $X \times Y$ satisfying
    \begin{gather*}
        \label{theorem: compactness of graph measures 1}
        \sup\{ \Gamma(K \times Y) \with \Gamma \in M \} < \infty, \\
        \label{theorem: compactness of graph measures 2}
        \lim_{\Gamma \in P} \sup \{ \Gamma(K \times (Y \without L)) \with \Gamma \in M \} = 0,
    \end{gather*}
    whenever $K$ is a compact subset of $X$, where $P$ is the family of compact subsets of $Y$ directed by the order $\preceq$ such that $L_1 \preceq L_2$ if and only if $L_1 \subset L_2$ for $L_1, L_2 \in P$.
    
    Then, $\closure M$ is compact, the push-forward $p_{\#}$ maps $\closure M$ continuously into $\mathscr{K}(X)^*$, and $p_\# \Gamma$ is a Radon measure over $X$ whenever $\Gamma \in \closure M$, where $p: X \times Y \to X$ is the projection map.
\end{theorem}
\begin{proof}
    Let $Z$ be the one-point compactification of $Y$, let $\iota: X \times Y \to X \times Z$ be the inclusion map, let $j$ be as in \ref{lemma: embedding of compact family of measures}, and let $M' = \{ \iota_\# \Gamma \with \Gamma \in M \}$. From \ref{lemma: embedding of compact family of measures}\eqref{lemma: embedding of compact family of measures 1}, the members of $M'$ are Radon measures over $X \times Z$. Applying \ref{lemma: compactness equivalence} with $M$ replaced by $M'$, we see that $\closure M'$ is a compact family of Radon measures $\phi$ over $X \times Z$ satisfying $\phi(K \times (Z \without Y)) = 0$ whenever $K$ is a compact subset of $X$;  thus, $\closure M = j^*[\closure M']$ is compact and $\iota_\# \circ j^* | \closure M' = \mathbf{1}_{\closure M'}$ by \ref{lemma: embedding of compact family of measures}\eqref{lemma: embedding of compact family of measures 2}\eqref{lemma: embedding of compact family of measures 3} applied with $M$ replaced by $\closure M'$. By \cite[2.2.17]{Federer69_MR0257325}, $q_\# (\iota_\# \Gamma) = p_\# \Gamma$ is a Radon measure over $X$ whenever $\Gamma \in \closure M$, where $q: X \times Z \to X$ is the projection map, and the continuity of $p_\#$ follows from the continuity of $(j^* | \closure M')^{-1}$.
\end{proof}

%----------------------------------------------------------------------------------------------------------

\begin{remark}
    \label{remark: stronger convergence}
    The proof also shows that $\iota_\#$ maps $\closure M$ continuously into $\mathscr{K}(X \times Z)^*$. In particular, if $X \times Y$ is second-countable (and hence metrizable by \cite[Chapter 4, Theorem 16]{Kelley75_MR0370454}) and $M$ consists of a convergent sequence $\Gamma_i$ and $\Gamma = \lim_{i \to \infty} \Gamma_i$, then we have $\iota_\# \Gamma = \lim_{i \to \infty} \iota_\# \Gamma_i$. It follows from \cite[1.62]{AFP00_MR1857292} that
    \[
        \lim_{i \to \infty} \textstyle\int g \ud \Gamma_i = \textstyle\int g \ud \Gamma
    \]
    for any bounded Borel function $g$ such that $\closure p[\spt g]$ is compact and the set of discontinuity points of g has $\Gamma$ measure zero.
\end{remark}

%----------------------------------------------------------------------------------------------------------

Next, we shall present a disintegration theorem that suits our setting, see also \cite[2.28]{AFP00_MR1857292}.

%----------------------------------------------------------------------------------------------------------

\begin{theorem}
    \label{theorem: disintegration}
    Suppose $X$ and $Y$ are second-countable locally compact Hausdorff spaces, $p: X \times Y \to X$ is the projection map, and $\Gamma$ is a Radon measure on $X \times Y$. If $\Gamma(K \times Y) < \infty$ for every compact subset $K$ of $X$, then $p_\# \Gamma$ is a Radon measure over $X$ and there exists a $p_{\#}\Gamma$ Young function $f$ of type $Y$ such that $\Gamma = \mathbf{Y}(p_{\#}\Gamma, f)$. Moreover, such a function $f$ is $p_\# \Gamma$ almost unique.
\end{theorem}
\begin{proof}
    Suppose $Z$ is the one-point compactification of $Y$, $\iota: Y \to Z$ is the inclusion map, and $q: X \times Z \to X$ is the projection map. Then $Z$ is a second-countable compact Hausdorff space and by \ref{lemma: embedding of compact family of measures}\eqref{lemma: embedding of compact family of measures 1}, $\mu = (\mathbf{1}_X \times \iota)_\# \Gamma$ is a Radon measure over $X \times Z$; in particular, $p_\# \Gamma = q_\# \mu$ is a Radon measure by \cite[2.2.17]{Federer69_MR0257325}. Clearly, we have $\mu(X \times (Z \without Y)) = 0$.
    
    By \ref{lemma: exponential law}, the function $S: \mathscr{K}(X, \mathscr{K}(Z)) \simeq \mathscr{K}(X \times Z)$ given by 
    \[
        S(u)(x, y) = u(x)(y) \quad \text{for $u \in \mathscr{K}(X, \mathscr{K}(Z))$}
    \]
    is a linear isomorphism and apply \cite[2.5.12]{Federer69_MR0257325} with $E = \mathscr{K}(Z)$ endowed with the sup-norm topology (which is separable by \cite[2.2, 2.23]{Menne16a_MR3528825}), $L = \mathscr{K}(X)$, $\Omega = \mathscr{K}(X, \mathscr{K}(Z))$, $T = \mu \circ S$, then there exists a Radon measure $\phi$ over $X$ and a $\phi$ measurable function $k$ with values in $\mathscr{K}(Z)^*$ with respect to the weak topology such that $|k(x)| = 1$ for $\phi$ almost all $x$ and such that
    \[
        \textstyle\int S(u) \ud \mu = T(u) = \textstyle\int \langle u(x), k(x) \rangle \ud \phi \, x \quad \text{whenever $u \in \mathscr{K}(X, \mathscr{K}(Z))$}.
    \]
    For $\beta \in \mathscr{K}(Z)^+$, we have
    \[
        0 \leq \textstyle\int \alpha(x)\beta(z) \ud \mu \, (x, z) = \textstyle\int \alpha(x) \langle \beta, k(x) \rangle \ud \phi \, x \quad \text{whenever $\alpha \in \mathscr{K}(X)^+$};
    \]
    thus, $\langle \beta, k(x) \rangle \geq 0$ for $\phi$ almost all $x$. As $\mathscr{K}(Z)^+$ is separable, we conclude for $\phi$ almost all $x$, $k(x)$ is monotone, hence $1 = |k(x)| = k(x)(Z)$. Therefore, $k$ is a $\phi$ Young function of type $Z$ and $\mu = \mathbf{Y}(\phi, k)$, hence $p_\# \Gamma = q_\# \mu = \phi$. By \ref{lemma: extension of disintegration formula to integrable functions}, we conclude that $k(x)(Z \without Y) = 0$ for $\phi$ almost all $x$.
    
    Note that extension by zero gives a continuous monomorphism 
    \[
        j: \mathscr{K}(Y) \to \mathscr{K}(Z),
    \]
    hence a continuous linear map 
    \[
        j^*: \mathscr{K}(Z)^* \to \mathscr{K}(Y)^*,
    \]
    where $\mathscr{K}(Z)^*$ and $\mathscr{K}(Y)^*$ are endowed with the weak topology. Therefore, $f = j^* \circ k$ is a $p_\# \Gamma$ Young function of type $Y$ such that
    \begin{align*}
        \textstyle\int \psi \ud \Gamma 
        &= \textstyle\int j(\psi(x, \cdot))(y) \ud \mu \, (x, y) \\
        &= \textstyle\int \langle j(\psi(x, \cdot)), k(x) \rangle \ud \phi \, x \\
        &= \textstyle\int \langle \psi(x, \cdot), f(x) \rangle \ud \phi \, x \\
        &= \textstyle\iint \psi(x, y) \ud f(x) \, y \ud (p_\# \Gamma) \, x
    \end{align*}
    whenever $\psi \in \mathscr{K}(X \times Y)$.

    Finally, suppose $g$ is a $p_\# \Gamma$ Young function of type $Y$ such that 
    \[
        \mathbf{Y}(p_\# \Gamma, g) = \mathbf{Y}(p_\# \Gamma, f).
    \]
    By \ref{lemma: two weak topologies agree} and \ref{lemma: embedding of compact family of measures}\eqref{lemma: embedding of compact family of measures 1}, we see $x \mapsto \iota_\#(g(x))$ for $x \in \domain g$ is a $p_\# \Gamma$ Young function of type $Z$. Note that $k(x)(Z \without Y) = 0$, hence by \ref{lemma: embedding of compact family of measures}\eqref{lemma: embedding of compact family of measures 2}, $\iota_\# (f(x)) = k(x)$ for $p_\# \Gamma$ almost all $x$. By \ref{lemma: extension of disintegration formula to integrable functions} and \cite[2.4.18]{Federer69_MR0257325}, we compute for $u \in \mathscr{K}(X, \mathscr{K}(Z))$
    \begin{align*}
        T(u)
        &= \textstyle\int \langle u(x), k(x) \rangle \ud (p_\# \Gamma) \, x \\
        &= \textstyle\iint S(u)(x, \iota(y) ) \ud f(x) \, y \ud (p_\# \Gamma) \, x \\
        &= \textstyle\int S(u)(x, \iota(y)) \ud \mathbf{Y}(p_\# \Gamma, f) \, (x, y) \\
        &= \textstyle\int S(u)(x, \iota(y)) \ud \mathbf{Y}(p_\# \Gamma, g) \, (x, y) \\
        &= \textstyle\iint S(u)(x, \iota(y) ) \ud g(x) \, y \ud (p_\# \Gamma) \, x \\
        &= \textstyle\iint S(u)(x, z) \ud \iota_\# (g(x)) \, z \ud (p_\# \Gamma) \, x \\
        &= \textstyle\int \langle u(x), \iota_\# (g(x)) \rangle \ud (p_\# \Gamma) \, x.
    \end{align*}
    From the uniqueness of $k$, we conclude $\iota_\# (g(x)) = k(x)$, hence by \ref{lemma: embedding of compact family of measures}\eqref{lemma: embedding of compact family of measures 2}, $g(x) = j^* (\iota_\# (g(x))) = (j^* \circ k)(x) = f(x)$ for $p_\# \Gamma$ almost all $x$.
\end{proof}

%----------------------------------------------------------------------------------------------------------

Finally, we will finish this section with the compactness theorem for pairs of varifolds and Young functions.

%----------------------------------------------------------------------------------------------------------

The following lemma shows that if $V$ is a rectifiable varifold, then there is a one-to-one correspondence between $\|V\|$ Young functions and $V$ Young functions obtained from precomposition with the functions $\tau$ and $p$ as in the lemma.

%----------------------------------------------------------------------------------------------------------

\begin{lemma}
    \label{lemma: young function correspondence}
    Suppose $m$ and $n$ are positive integers, $m \leq n$, $V$ is an $m$-dimensional varifold in some open set $U$ of $\mathbf{R}^n$, and define $\tau$ by 
    \[
        \tau(x) = (x, \tangent^m(\|V\|, x)) \quad \text{whenever $\tangent^m(\|V\|, x) \in \mathbf{G}(n, m)$.}
    \]
    Then, the following statements hold.
    \begin{enumerate}
        \item 
            \label{lemma: young function correspondence 1}
            $V = \tau_\# \|V\|$.
        \item 
            \label{lemma: young function correspondence 2}
            For $A \subset U \times \mathbf{G}(n, m)$, $A$ is $V$ measurable if and only if $\tau^{-1}[A]$ is $\|V\|$ measurable.
        \item 
            \label{lemma: young function correspondence 3}
            $\tau(p(x, S)) = (x, S)$ for $V$ almost all $(x, S)$, where $p: U \times \mathbf{G}(n, m) \to U$ is the projection map.
    \end{enumerate}
\end{lemma}
\begin{proof}
    By \cite[2.11]{MS25_MR4864996} and \cite[3.5(1b)]{Allard72_MR307015}, $\tau_{\#} \|V\|$ is a Radon measure and $V = \tau_{\#} \|V\|$. Let $A \subset U \times \mathbf{G}(n, m)$. Clearly, if $\tau^{-1}[A]$ is $\|V\|$ measurable, then $A$ is $V = \tau_\# \|V\|$ measurable. If $A$ is $V$ measurable, then there exists a Borel set $B$ such that $A \subset B$ and $V(B \without A) = 0$, hence $\|V\| \tau^{-1}[B \without A] = 0$. Therefore, 
    \[
        \tau^{-1}[A] = \tau^{-1}[B] \without \tau^{-1}[B \without A]
    \]
    is $\|V\|$ measurable and \eqref{lemma: young function correspondence 2} follows. To prove \eqref{lemma: young function correspondence 3}, we note that 
    \[
        \tau^{-1}[\{ (x, S) \with (x, S) \neq (x, \tangent^m(\|V\|, x)) \}] = \varnothing,
    \]
    and the assertion follows from \eqref{lemma: young function correspondence 1}.
\end{proof}

%----------------------------------------------------------------------------------------------------------

\begin{theorem}[Compactness]
    \label{theorem: compactness theorem for pairs of rectifiable varifolds and Young functions}
    Suppose $m$ and $n$ are positive integers, $m \leq n$, $V_i$ is a sequence of $m$-dimensional varifolds in some open set $U$ of $\mathbf{R}^n$ such that $V = \lim_{i \to \infty} V_i$,
    $Y$ is a second-countable locally compact Hausdorff space, $L_j$ is a sequence of compact subsets of $Y$ such that $\bigcup_{j=1}^\infty L_j = Y$ and $L_j \subset \interior L_{j+1}$ whenever $j \in \mathscr{P}$, and $f_i$ is a sequence of $\|V_i\|$ Young functions of type $Y$ such that
    \[
        \lim_{j \to \infty} \sup \{ \mathbf{Y}(\|V_i\|, f_i)(K \times (Y \without L_j)) \with i \in \mathscr{P} \} = 0,
    \]
    whenever $K$ is a compact subset of $U$.

    Then, exists a $\|V\|$ Young function $f$ of type $Y$ such that, possibly passing to a subsequence,
    \[
        \mathbf{Y}(\|V\|, f) = \lim_{i \to \infty} \mathbf{Y}(\|V_i\|, f_i);
    \]
    moreover, if $V_i$ and $V$ are rectifiable, then
    \[
        \mathbf{Y}(V, f \circ p) = \lim_{i \to \infty} \mathbf{Y}(V_i, f_i \circ p),
    \]
    where $p: U \times \mathbf{G}(n, m) \to U$ is the projection map.
\end{theorem}
\begin{proof}
    The main assertion follows from \ref{theorem: compactness of graph measures} and \ref{theorem: disintegration}, and the postscript follows from \ref{lemma: young function correspondence}.
\end{proof}

%%%%%%%%%%%%%%%%%%%%%%%%%%%%%%%%%%%%%%%%%%%%%%%%%%%%%%%%%%%%%%%%%%%%%%%%%%%%%%%%%%%%%%%%%%%%%%%%%%%%%%%%%%%
\section{The space of probability Radon measures and operations on Young functions}
\label{section: the space of probability Radon measures and operations on Young functions}
%%%%%%%%%%%%%%%%%%%%%%%%%%%%%%%%%%%%%%%%%%%%%%%%%%%%%%%%%%%%%%%%%%%%%%%%%%%%%%%%%%%%%%%%%%%%%%%%%%%%%%%%%%%

In this section, we first define a metric on the space of probability Radon measures and study
its topological and metric structures. Then, we show that the class of Young functions is stable under several operations.

%----------------------------------------------------------------------------------------------------------

\begin{definition}
    \label{definition: function space E}
    Suppose $Y$ is a finite-dimensional Banach space. For $0 \leq s < \infty$, the space $\mathbf{E}_s(Y)$ consists of functions $\gamma: Y \to \mathbf{R}$ of class $1$ such that $\gamma(0) = 0$ and $\spt \der \gamma \subset \mathbf{B}(0, s)$. We endow $\mathbf{E}_s(Y)$ with the norm whose value equals
    \[
        \sup \im \| \der \gamma \| \quad \text{at $\gamma \in \mathbf{E}_s(Y)$},
    \]
    and we endow $\mathbf{E}(Y) = \bigcup \{ \mathbf{E}_s(Y) \with 0 \leq s < \infty \}$ with the locally convex final topology induced by the inclusion maps $\mathbf{E}_s(Y) \to \mathbf{E}(Y)$.
\end{definition}

%----------------------------------------------------------------------------------------------------------

\begin{remark}
    Suppose $Y$ is a finite-dimensional Banach space. Letting $k = \dimension Y$ and employing the linear isomorphism $Y \simeq \mathbf{R}^k$, the standard mollification argument in Euclidean spaces also works for $Y$.
\end{remark}

%----------------------------------------------------------------------------------------------------------

\begin{remark}
    \label{remark: basic estimate E_s}
    Whenever $0 \leq s < \infty$ and $\gamma \in \mathbf{E}_s(Y)$, we have
    \[
        |\gamma(y)| \leq \inf\{ |y|, s \} \sup \im \|\der \gamma\| \quad \text{for $y \in Y$};
    \]
    in particular, $\gamma$ is bounded.
\end{remark}

%----------------------------------------------------------------------------------------------------------

\begin{remark}
    \label{remark: probability measures as dual elements}
    The canonical map $\mathbf{P}(Y) \to \mathbf{E}(Y)^*$ is continuous by \ref{lemma: two weak topologies agree} and \ref{remark: basic estimate E_s}; so is its composition with the dual map $\mathbf{E}(Y)^* \to \mathbf{E}_s(Y)^*$ whenever $0 \leq s < \infty$, where the dual spaces are endowed with the corresponding weak topology.
\end{remark}

The following lemma shows that, with the help of Ascoli's theorem, the condition on the support of functions in $\mathbf{E}_s(Y)$ allows us to retain the continuity of the canonical map $\mathbf{P}(Y) \to \mathbf{E}_s(Y)^*$ with respect to the dual norm topology on $\mathbf{E}_s(Y)^*$.

%----------------------------------------------------------------------------------------------------------

\begin{lemma}
    \label{lemma: probability measures as dual elements}
    Suppose $Y$ is a finite-dimensional Banach space and $0 \leq s < \infty$. Then, the canonical map $\mathbf{P}(Y) \to \mathbf{E}_s(Y)^*$ is continuous with bounded image, where $\mathbf{E}_s(Y)^*$ is endowed with the dual norm topology.
\end{lemma}
\begin{proof}
    By \cite[2.10.21]{Federer69_MR0257325}, the function space 
    \[
        \mathbf{E}_s(Y) \cap \{ \gamma \with \sup \im \|\der \gamma\| \leq 1 \}
    \]
    is contained in the compact space $F$ of Lipschitzian functions $f: Y \to \mathbf{R}$ with $f(0) = 0$ and $\lipschitz f \leq 1$, where $F$ is endowed with the topology of uniform convergence on each compact subset of $Y$, or equivalently, the topology induced by the metric $\rho$ with value
    \[
        \rho(f, g) = \sum_{i=1}^\infty 2^{-i} \cdot \frac{\sup \im |(f - g)| \mathbf{B}(0, i)|}{1 + \sup \im |(f - g)| \mathbf{B}(0, i)|} \quad \text{at $(f, g) \in F \times F$}.
    \]
    Note that
    \[
        \sup \im |\gamma_1 - \gamma_2| = \sup \{ |\gamma_1(y) - \gamma_2(y)| \with y \in Y \cap \mathbf{B}(0, s) \}
    \]
    whenever $\gamma_1, \gamma_2 \in \mathbf{E}_s(Y)$. Then, it is straightforward to show that the function space $\mathbf{E}_s(Y) \cap \{ \gamma \with \sup \im \|\der \gamma\| \leq 1 \}$ is totally bounded with respect to the supremum metric. With the help of \ref{lemma: two weak topologies agree}, it follows that every convergent sequence in $\mathbf{P}(Y)$ also converges in $\mathbf{E}_s(Y)^*$. Since $\mathbf{P}(Y) \subset \mathscr{K}(Y)^*$ is metrizable by \cite[2.23]{Menne16a_MR3528825}, the continuity of the canonical map follows, and the boundedness of its image is immediate by \ref{remark: basic estimate E_s}.
\end{proof}

%----------------------------------------------------------------------------------------------------------

\begin{lemma}
    \label{lemma: isometric embedding}
    Suppose $Y$ is a finite-dimensional Banach space. There holds
    \[
        |y_1 - y_2| = \sup \{ \gamma(y_1) - \gamma(y_2) \with \gamma \in \mathbf{E}(Y), \sup \im \|\der \gamma\| \leq 1 \}
    \]
    whenever $y_1, y_2 \in Y$.
\end{lemma}
\begin{proof}
    Since $|\gamma(y_1) - \gamma(y_2)| \leq |y_1 - y_2| \sup \im \|\der \gamma\|$ whenever $\gamma \in \mathbf{E}(Y)$, it suffices to show there exists a sequence $\gamma_i \in \mathbf{E}(Y)$ such that 
    \[
        \lim_{i \to \infty} \gamma_i(y_1) - \gamma_i(y_2) = |y_1 - y_2|.
    \]
    Note that there exists $L \in \homomorphism(Y, \mathbf{R})$ such that $L(y_1 - y_2) = |y_1 - y_2|$ and $\|L\| \leq 1$.
    Whenever $i \in \mathscr{P}$, since $L_i = \sup \{ \inf \{ L, i \}, -i \}$ is uniformly continuous and $L_i(0) = 0$, there exists $\beta_i: Y \to \mathbf{R}$ of class $1$ such that $\beta_i(0) = 0$, $\sup \im \|\der \beta_i\| \leq 1$, and $\sup \im |L_i - \beta_i| \leq i^{-1}$. Choose a function $\phi: Y \to \mathbf{R}$ of class $1$ such that $0 \leq \phi \leq 1$, $\phi| \mathbf{B}(0, 1) = 1$, $\sup \im \|\der \phi\| \leq 1$, and $\spt \phi$ is compact, and we define $\phi_i(y) = (1 - i^{-1}) \phi(i^{-3}y)$ whenever $y \in Y$ and $i \in \mathscr{P}$. Then, we have $\gamma_i = \phi_i \beta_i \in \mathbf{E}(Y)$, $\sup \im \|\der\gamma_i\| \leq 1$ for $i \geq 1$, and $L(y) = \lim_{i \to \infty}\gamma_i(y)$ for all $y \in Y$ and the assertion follows.
\end{proof}

%----------------------------------------------------------------------------------------------------------

\begin{definition}
    \label{definition: pseudo-metric d}
    Suppose $Y$ is a finite-dimensional Banach space, we define the pseudo-metric $d$ on $\mathbf{P}(Y)$ by
    \[
        d(\mu, \nu) = \sup \left\{ \textstyle\int \gamma \ud \mu - \textstyle\int \gamma \ud \nu \with \gamma \in \mathbf{E}(Y), \sup \im \|\der \gamma\| \leq 1 \right\}
    \]
    whenever $\mu, \nu \in \mathbf{P}(Y)$.
\end{definition}

%----------------------------------------------------------------------------------------------------------

\begin{remark}
    If $\mu, \nu \in \mathbf{P}(Y)$ satisfy $d(\mu, \nu) = 0$, we have $\int \gamma \ud \mu = \int \gamma \ud \nu$ whenever $\gamma \in \mathscr{K}(Y)$ with $\lipschitz \gamma < \infty$; it follows that $\mu = \nu$.
\end{remark}

%----------------------------------------------------------------------------------------------------------

\begin{remark}
    \label{remark: equivalent def of d}
    By approximation, we readily show that
    \begin{align*}
        d(\mu, \nu) 
        &= \sup \left\{ \textstyle\int \gamma \ud \mu - \textstyle\int \gamma \ud \nu \with \gamma: Y \to \mathbf{R}, \lipschitz \gamma \leq 1 \right\} \\
        &= \sup \left\{ \textstyle\int \gamma \ud \mu - \textstyle\int \gamma \ud \nu \with \gamma \in \mathscr{K}(Y), \lipschitz \gamma \leq 1 \right\}.
    \end{align*}
    whenever $\mu, \nu \in \mathbf{P}(Y)$.
\end{remark}

%----------------------------------------------------------------------------------------------------------

\begin{remark}
    \label{remark: isometric embedding}
    It follows from \ref{lemma: isometric embedding} that $\boldsymbol{\updelta}: Y \to \mathbf{P}(Y)$ as in \ref{remark: single-valued Young function} is an isometric embedding.
\end{remark}

%----------------------------------------------------------------------------------------------------------

\begin{definition}
    \label{definition: Lipschitz continuity Young function}
    Suppose $(X, \rho)$ is a pseudo-metric space and $Y$ is a finite-dimensional Banach space. A function $f: X \to \mathbf{P}(Y)$ is termed to be \textit{Lipschitzian} if and only if there exists $0 < M < \infty$ such that
    \[
        d(f(x), f(y)) \leq M \rho(x, y) \quad \text{whenever $x, y \in X$},
    \]
    where $d$ as in \ref{definition: pseudo-metric d}. The Lipschitz constant of $f$ is defined to be the infimum of all such numbers $M$. We say $f$ is \textit{locally Lipschitzian} if and only if whenever $x \in X$, there exists a neighborhood $U$ of $x$ such that $f|U$ is Lipschitzian.
\end{definition}

%----------------------------------------------------------------------------------------------------------

\begin{remark}
    Note that if $X$ is a metric space and $f$ is a Lipschitzian $\mathbf{P}(Y)$-valued function, then we have $d(f(x), f(y)) < \infty$ whenever $x, y \in X$; in particular, $d| (\im f \times \im f)$ is a metric and $f$ is Lipschitzian in the usual sense.
\end{remark}

%----------------------------------------------------------------------------------------------------------

\begin{remark}
    \label{remark: distance and 1-moment}
    By \ref{remark: equivalent def of d}, we have
    \[
        d(\boldsymbol{\updelta}_0, \mu) = \textstyle\int |y| \ud \mu \, y \quad \text{whenever $\mu \in \mathbf{P}(Y)$}.
    \]
    Let
    \[
        \mathbf{P}_1(Y) = \mathbf{P}(Y) \cap \left\{ \mu \with \textstyle\int |y| \ud \mu \, y < \infty \right\}.
    \]
    Then, $(\mathbf{P}_1(Y)), d)$ is a metric space.
\end{remark}

%----------------------------------------------------------------------------------------------------------

Next, we shall investigate the relation between the weak topology and the $d$ topology on $\mathbf{P}_1(Y)$ when $Y$ is a finite-dimensional Banach space.

%----------------------------------------------------------------------------------------------------------

\begin{remark}
    \label{remark: d agrees with 1-Wasserstein metric}
    The notation $\mathbf{P}_1(Y)$ is taken from \cite[6.4]{Vil09_MR2459454}, on which the metric $d$ agree with $1$-Wasserstein metric, see \cite[6.1, 6.5]{Vil09_MR2459454}. The following characterization of $d$ convergence on $\mathbf{P}_1(Y)$ is a special case of \cite[6.9]{Vil09_MR2459454}.
\end{remark}

%----------------------------------------------------------------------------------------------------------

\begin{lemma}[see \protect{\cite[6.8, 6.9]{Vil09_MR2459454}}]
    \label{lemma: characterization of the metric topology}
    Suppose $Y$ is a finite-dimensional Banach space, $\mu_i$ is a sequence in $\mathbf{P}_1(Y)$, and $\mu \in \mathbf{P}_1(Y)$. Then, the following two statements are equivalent.
    \begin{enumerate}
        \item 
            $\lim_{i \to \infty} d(\mu_i, \mu) = 0$.
        \item 
            $\mu_i \to \mu$ weakly as $i \to \infty$ and $\lim_{i \to \infty} \int |y| \ud \mu_i \, y = \int |y| \ud \mu \, y$.
    \end{enumerate}
\end{lemma}
\begin{proof}
    Suppose $\lim_{i \to \infty} d(\mu_i, \mu) = 0$, then
    \[
        \lim_{i \to \infty} \textstyle\int |y| \ud \mu_i \, y = \lim\limits_{i \to \infty} d(\boldsymbol{\updelta}_0, \mu_i) = d(\boldsymbol{\updelta}_0, \mu) = \int |y| \ud \mu \, y.
    \]
    Whenever $\varepsilon > 0$ and $f \in \mathscr{K}(Y)$, since $f$ is uniformly continuous, there exists a function $\gamma: Y \to \mathbf{R}$ of class $1$ such that $\spt \gamma$ is compact and $\sup \im |f - \gamma| < \varepsilon$, hence we have
    \begin{align*}
        |\mu_i(f) - \mu(f)| \leq 2 \varepsilon + |\mu_i(\gamma) - \mu(\gamma)| \leq 2\varepsilon + \sup\im \|\der \gamma\| d(\mu_i, \mu).
    \end{align*}
    Therefore, we conclude $\mu_i \to \mu$ weakly as $i \to \infty$.
    
    To prove the converse, letting $\varepsilon > 0$, we will first show that there exists $0 \leq s < \infty$ such that
    $\textstyle\int_{Y \without \mathbf{B}(0, s)} |y| \ud \mu \, y \leq \varepsilon$ and
    \[
        \textstyle\int_{Y \without \mathbf{B}(0, s)} |y| \ud \mu_i \, y \leq \varepsilon \quad \text{whenever $i \in \mathscr{P}$}.
    \]
    Since $\mu$ is a Radon measure, there exists $f \in \mathscr{K}(Y)$ such that $0 \leq f \leq 1$ and $\textstyle\int (1 - f(y)) |y| \ud \mu \, y \leq \varepsilon/2$. From the hypothesis, there exists $N \in \mathscr{P}$ such that 
    \[
        \left| \textstyle\int f(y)|y| \ud \mu_i \, y - \int f(y)|y| \ud \mu \, y \right| + \left| \textstyle\int |y| \ud \mu_i \, y - \int |y| \ud \mu \, y \right| \leq \varepsilon/2
    \]
    whenever $i \geq N$. Then, we have
    \[
        \textstyle\int_{Y \without \spt f} |y| \ud \mu_i \, y \leq \varepsilon \quad \text{whenever $i \geq N$}
    \]
    and the assertion follows. On the other hand, by \cite[2.10.21]{Federer69_MR0257325}, the set $\mathscr{K}_{\mathbf{B}(0, r)}(Y) \cap \{ \gamma \with \lipschitz\gamma \leq 1 \}$ is compact in the Banach space $\mathscr{K}_{\mathbf{B}(0, r)}(Y)$; since each member of $\mathbf{P}(Y)$ defines a continuous linear functional over $\mathscr{K}_{\mathbf{B}(0, r)}$, we have
    \[
        \lim_{i \to \infty} \sup \{ \textstyle\int \gamma \ud \mu_i - \int \gamma \ud \mu \with \gamma \in \mathscr{K}_{\mathbf{B}(0, r)}(Y), \lipschitz \gamma \leq 1 \} = 0
    \]
    whenever $0 < r < \infty$. It follows that $\lim_{i \to \infty} d(\mu_i, \mu) = 0$.
\end{proof}

%----------------------------------------------------------------------------------------------------------

\begin{theorem}
    The map $(\mathbf{P}_1(Y), d) \to \mathbf{P}(Y) \times \mathbf{R}$ defined by 
    \[
        \mu \mapsto (\mu, \textstyle\int |y| \ud \mu \, y) \quad \text{for $\mu \in \mathbf{P}_1(Y)$}
    \]
    is a homeomorphic embedding.
\end{theorem}
\begin{proof}
    Since the weak topology on the space of Radon measures over $Y$ is metrizable by \cite[2.23]{Menne16a_MR3528825}, the assertion follows from \ref{lemma: characterization of the metric topology}.
\end{proof}

%----------------------------------------------------------------------------------------------------------

\begin{remark}
    We can not replace $\mathbf{P}_1(Y)$ with $\mathbf{P}(Y)$ in \ref{lemma: characterization of the metric topology}. To construct a counterexample, taking $Y = \mathbf{R}$, $\mu \in \mathbf{P}(\mathbf{R}) \without \mathbf{P}_1(\mathbf{R})$, and $\mu_i = \mu \restrict \mathbf{B}(0, i)$, it is easy to verify $\mu_i \to \mu$ weakly and $\int |y| \ud \mu_i \, y \uparrow \int |y| \ud \mu \, y$ as $i \to \infty$. Since $\mu_i \in \mathbf{P}_1(\mathbf{R})$ and $\mu \in \mathbf{P}(\mathbf{R}) \without \mathbf{P}_1(\mathbf{R})$, we have $d(\mu_i, \boldsymbol{\updelta}_0) < \infty$ and $d(\mu, \boldsymbol{\updelta}_0) = \infty$ by \ref{remark: distance and 1-moment}; it follows that $d(\mu_i, \mu) = \infty$ for all $i \in \mathscr{P}$.
\end{remark}

%----------------------------------------------------------------------------------------------------------

\begin{remark}
    The $d$ topology is strictly finer than the weak topology on $\mathbf{P}_1(Y)$ provided $\dimension Y \geq 1$; for instance, letting $0 \neq v \in Y$ and $\mu \in \mathbf{P}_1(Y)$, the sequence $\mu_i = (1 - i^{-1})\mu + i^{-1}\boldsymbol{\updelta}_{i v}$ converges weakly to $\mu$ as $i \to \infty$ but 
    \[
        \lim_{i \to \infty}\textstyle\int |y| \ud \mu_i \, y - \textstyle\int |y| \ud \mu \, y = |v| > 0.
    \]
    
    However, these two topologies induce the same subspace topology on the weakly closed subsets
    \[
        A(\varepsilon) = \mathbf{P}(Y) \cap \{ \mu \with \text{$\mu\{ a \} \geq \varepsilon$ whenever $a \in \spt \mu$} \} \subset \mathbf{P}_1(Y)
    \]
    whenever $\varepsilon > 0$. To show this, suppose $\varepsilon > 0$ and $\mu_i$ is a sequence in $A(\varepsilon)$ that converges weakly to some $\mu \in \mathbf{P}(Y)$. Whenever $0 < r < \infty$ and $a \in \spt \mu$, by \cite[2.6(2c)]{Allard72_MR307015}, there exists $N \in \mathscr{P}$ such that $\mu_i(\mathbf{U}(a, r)) > 0$, hence $\mu_i(\mathbf{U}(a, r)) \geq \varepsilon$, whenever $i \geq N$; it follows that 
    \[
        \mu(\mathbf{B}(a, r)) \geq \limsup_{i \to \infty} \mu_i(\mathbf{B}(a, r)) \geq \varepsilon.
    \]
    Thus, we conclude $\mu \in A(\varepsilon)$. Since $\lim_{i \to \infty}\int (1 - f) \ud \mu_i = \int (1 - f) \ud \mu$ whenever $f \in \mathscr{K}(Y)$ and note that $\nu(S) > 0$ implies $\nu(S) \geq \varepsilon$ whenever $S \subset Y$ and $\nu \in A(\varepsilon)$, there exists $0 < s < \infty$ such that
    \[
        \sup \{ \mu_i(Y \without \mathbf{B}(0, s)) \with i \in \mathscr{P} \} = 0 \quad \text{for some $0 < s < \infty$};
    \]
    by \ref{lemma: two weak topologies agree},
    \[
        \lim_{i \to \infty} \textstyle\int |y| \ud \mu_i \, y = \lim\limits_{i \to \infty} \textstyle\int \inf\{ |y|, s \} \ud \mu_i \, y = \textstyle\int \inf\{ |y|, s \} \ud \mu \, y = \textstyle\int |y| \ud \mu \, y.
    \]
\end{remark}

%----------------------------------------------------------------------------------------------------------

The following theorem shows that the class of Young functions is stable under pushforward by Borel functions pointwise and taking the Cartesian product of measures pointwise.

%----------------------------------------------------------------------------------------------------------

\begin{theorem}
    \label{theorem: push-forward and Cartesian product of Young functions}
    Suppose $X$ and $Y$ are second-countable locally compact Hausdorff spaces and $f: X \to Y$ is a Borel function. Then, the following two statements hold.
    \begin{enumerate}
        \item 
        \label{theorem: push-forward and Cartesian product of Young functions 1}
            The pushforward of measures
            \[
                (f_\# \mu)(A) = \mu(f^{-1}[A]) \quad \text{whenever $\mu \in \mathbf{P}(X)$ and $A \subset Y$},
            \]
            defines a Borel function $f_\#: \mathbf{P}(X) \to \mathbf{P}(Y)$. Moreover, if $f$ is continuous, then so is $f_\#$.
        \item 
        \label{theorem: push-forward and Cartesian product of Young functions 2}
            The product of measures
            \[
                (\mu, \nu) \mapsto \mu \times \nu \quad \text{whenever $\mu \in \mathbf{P}(X)$ and $\nu \in \mathbf{P}(Y)$}
            \]
            defines a continuous function $\mathbf{P}(X) \times \mathbf{P}(Y) \to \mathbf{P}(X \times Y)$.
    \end{enumerate}
\end{theorem}
\begin{proof}
    We first prove \eqref{theorem: push-forward and Cartesian product of Young functions 1}. By \cite[Chapter 4, Theorem 16]{Kelley75_MR0370454}, \cite[2.11]{MS25_MR4864996}, and \cite[2.2.2]{Federer69_MR0257325}, $f_\# \mu$ is a Radon measure over $Y$. Note that it is enough to show $\mu \mapsto \int g \ud (f_\#\mu)$ is a Borel function whenever $g \in \mathscr{K}(Y)$ by \cite[2.23]{Menne16a_MR3528825}. Since $\int g \ud (f_\#\mu) = \int (g \circ f) \ud \mu$ whenever $\mu \in \mathbf{P}(X)$ and $g: Y \to \mathbf{R}$ is a bounded Borel function by \cite[2.4.18]{Federer69_MR0257325}, it reduces to show $\mu \mapsto \int h \ud \mu$ is a Borel function whenever $h: X \to \mathbf{R}$ is a bounded Borel function.
    
    For $i \in \mathscr{P}$, let $B_i$ be the family of all Borel functions $h: X \to \mathbf{R}$ with $\sup \im |h| \leq i$ such that $\mu \mapsto \int h \ud \mu$ is a Borel function on $\mathbf{P}(X)$. By \ref{lemma: two weak topologies agree}, the set $B_i$ contains all continuous functions from $X$ into $\mathbf{R} \cap \{ r \with |r| \leq i \}$. By \cite[2.4.9]{Federer69_MR0257325} and employing the terminology of \cite[2.2.15]{Federer69_MR0257325}, $B_i$ is a Baire class, hence contains all Baire functions with image contained in $\mathbf{R} \cap \{ r \with |r| \leq i \}$; applying \cite[2.2.15]{Federer69_MR0257325} with $Y = \mathbf{R} \cap \{ r \with |r| \leq i \}$, we see $B_i$ contains all Borel functions $h: X \to \mathbf{R}$ with $\sup \im |g| \leq i$. Therefore, the main conclusion of \eqref{theorem: push-forward and Cartesian product of Young functions 1} follows. For the postscript of \eqref{theorem: push-forward and Cartesian product of Young functions 1}, it is enough to show $\mu \mapsto \int g \ud (f_\#\mu) = \int (g \circ f) \ud \mu$ is continuous whenever $g \in \mathscr{K}(Y)$, which is immediate from \ref{lemma: two weak topologies agree}.

    Next, we will show \eqref{theorem: push-forward and Cartesian product of Young functions 2}. By Fubini's theorem \cite[2.6.2]{Federer69_MR0257325}, $\mu \times \nu$ is a Borel regular measure, hence by \cite[2.2.2]{Federer69_MR0257325}, a Radon measure whenever $\mu \in \mathbf{P}(X)$ and $\nu \in \mathbf{P}(Y)$. It is enough to show that $e_\phi(\mu, \nu) = \textstyle\int \phi \ud (\mu \times \nu)$ defines a continuous function on $\mathbf{P}(X) \times \mathbf{P}(Y)$ whenever $\phi \in \mathscr{K}(X \times Y)$. If $\phi \in \mathscr{K}(X \times Y)$ satisfies $\phi(x, y) = \alpha(x)\beta(y)$ whenever $(x, y) \in X \times Y$ for some $\alpha \in \mathscr{K}(X)$ and $\beta \in \mathscr{K}(Y)$, then $e_\phi$ is continuous. Therefore, in view of \ref{lemma: K(X) tensor K(Y) density}, for each $\phi \in \mathscr{K}(X \times Y)$, $e_\phi$ is the uniform limit of a sequence of continuous functions, hence itself a continuous function.
\end{proof}

%----------------------------------------------------------------------------------------------------------

\begin{remark}
    The two statements in \ref{theorem: push-forward and Cartesian product of Young functions} are variants of \cite[17.28, 17.40]{Kec95_MR1321597}.
\end{remark}

%----------------------------------------------------------------------------------------------------------

\begin{definition}
    \label{definition: pushforward of Young function}
    Suppose $Y$ and $Z$ are second-countable locally compact Hausdorff spaces, $\mu$ is a Radon measure over a locally compact Hausdorff space $X$, $f: Y \to Z$ is a Borel function, and $g$ is a $\mu$ Young function of type $Y$. Then, we define the $\mu$ Young function $f_\# g$ of type $Z$ by $(f_\# g)(x) = f_\#(g(x))$ whenever $x \in \domain g$.
\end{definition}

%----------------------------------------------------------------------------------------------------------

\begin{definition}
    \label{definition: product of Young functions}
    Suppose $X$, $Y$, and $Z$ are locally compact Hausdorff spaces, $Y$ and $Z$ are second-countable, $\mu$ is a Radon measure over $X$, and $f$ and $g$ are $\mu$ Young functions of types $Y$ and $Z$, respectively. Then, we define the $\mu$ Young function $f \times g$ of type $Y \times Z$ by $(f \times g)(x) = f(x) \times g(x)$ whenever $x \in \domain f \cap \domain g$.
\end{definition}

%----------------------------------------------------------------------------------------------------------

We finish this section by introducing the sum of two vector-valued functions and the product of two real-valued functions in the context of Young functions.

%----------------------------------------------------------------------------------------------------------

\begin{definition}[see \protect{\cite[14.20]{Kle20_MR4201399}}]
    \label{definition: convolution of measures}
    Suppose $\mu$ and $\nu$ are measures over a vector space $Y$, and $A: Y \times Y \to Y$ is defined by
    \[
        A(x, y) = x + y \quad \text{whenever $x, y \in Y$}.
    \]
    Then, the measure $\mu * \nu = A_{\#} (\mu \times \nu)$ is termed the \textit{convolution of $\mu$ and $\nu$}.
\end{definition}

%----------------------------------------------------------------------------------------------------------

\begin{remark}[see \protect{\cite[14.21]{Kle20_MR4201399}}]
    \label{remark: properties of convolution}
    Suppose $Y$ is a finite-dimensional Banach space. We readily verify the following basic properties of convolutions of measures.
    \begin{enumerate}
        \item 
            \label{remark: properties of convolution1}
            Whenever $\mu, \nu \in \mathbf{P}(Y)$, we have $\mu * \nu = \nu * \mu$.
        \item 
            \label{remark: properties of convolution2}
            If $\mu, \nu \in \mathbf{P}(Y)$, then $\mu * \nu \in \mathbf{P}(Y)$.
        \item 
            \label{remark: properties of convolution3}
            If $\mu, \nu \in \mathbf{P}(Y)$ satisfy $\mu = \sum_{x \in A} f(x) \boldsymbol{\updelta}_{x}$ and $\nu = \sum_{y \in B} g(y) \boldsymbol{\updelta}_{y}$ for some finite subsets $A$ and $B$ of $Y$ and functions $f: A \to \mathbf{R}$ and $g: B \to \mathbf{R}$, then $\mu * \nu = \sum_{(x, y) \in A \times B} f(x) g(y) \boldsymbol{\updelta}_{x + y}$.
    \end{enumerate}
\end{remark}

%----------------------------------------------------------------------------------------------------------

\begin{corollary}
    \label{corollary: convolution of Young functions}
    Suppose $Y$ is a finite-dimensional Banach space and $d$ is as in \ref{definition: pseudo-metric d}. Then, the convolution of measures
    \[
        (\mu, \nu) \mapsto \mu * \nu \quad \text{whenever $\mu, \nu \in \mathbf{P}(Y)$}
    \]
    defines a continuous function $\mathbf{P}(Y) \times \mathbf{P}(Y) \to \mathbf{P}(Y)$ with respect to the weak topologies. Furthermore, we have 
    \[
        d(\mu * \nu, \lambda * \eta) \leq d(\mu, \lambda) + d(\nu, \eta) \quad \text{whenever $\mu, \nu, \lambda, \eta \in \mathbf{P}(Y)$};
    \]
    in particular, the convolution of measures defines a Lipschitzian map $(\mathbf{P}(Y) \times \mathbf{P}(Y), \rho) \to (\mathbf{P}(Y), d)$ where $\rho$ satisfies $\rho((\mu, \nu), (\lambda, \eta)) = d(\mu, \nu) + d(\lambda, 
    \eta)$ whenever $\mu, \nu, \lambda, \eta \in \mathbf{P}(Y)$.
\end{corollary}
\begin{proof}
    The main assertion follows from \ref{theorem: push-forward and Cartesian product of Young functions}\eqref{theorem: push-forward and Cartesian product of Young functions 1}\eqref{theorem: push-forward and Cartesian product of Young functions 2}. The postscript follows from \ref{remark: equivalent def of d}, \ref{remark: properties of convolution}\eqref{remark: properties of convolution1}, and the estimate
    \begin{align*}
        \textstyle\int \gamma \ud (\mu * \lambda) - \textstyle\int \gamma \ud (\nu * \lambda) 
        &= \textstyle\int \left(\int \gamma(x + y) \ud \mu \, x - \int \gamma(x + y) \ud \nu \, x \right) \ud \lambda \, y \\
        &\leq d(\mu, \nu)
    \end{align*}
    whenever $\mu, \nu, \lambda \in \mathbf{P}(Y)$ and $\gamma \in \mathscr{K}(Y)$ with $\lipschitz \gamma \leq 1$.
\end{proof}

%----------------------------------------------------------------------------------------------------------

\begin{remark}
    Similarly, if $Y = \mathbf{R}$, replacing $A$ with the multiplication on $\mathbf{R}$ induces a continuous function $\mathbf{P}(\mathbf{R}) \times \mathbf{P}(\mathbf{R}) \to \mathbf{P}(\mathbf{R})$.
\end{remark}

%----------------------------------------------------------------------------------------------------------

\begin{definition}
    \label{definition: convolution of Young functions}
    Suppose $X$ is a locally compact Hausdorff space, $Y$ is a finite-dimensional Banach space, $\mu$ is a Radon measure over $X$, and $f$ and $g$ are $\mu$ Young functions of type $Y$. Then, we define the $\mu$ Young function $f * g$ by $(f * g)(x) = f(x) * g(x)$ whenever $x \in \domain f \cap \domain g$.
\end{definition}

%----------------------------------------------------------------------------------------------------------

\begin{remark}
    By the postscript of \ref{corollary: convolution of Young functions}, we have $\lipschitz f*g \leq \lipschitz f + \lipschitz g$.
\end{remark}

%%%%%%%%%%%%%%%%%%%%%%%%%%%%%%%%%%%%%%%%%%%%%%%%%%%%%%%%%%%%%%%%%%%%%%%%%%%%%%%%%%%%%%%%%%%%%%%%%%%%%%%%%%%
\section{The test function spaces}
\label{section: the test function spaces}
%%%%%%%%%%%%%%%%%%%%%%%%%%%%%%%%%%%%%%%%%%%%%%%%%%%%%%%%%%%%%%%%%%%%%%%%%%%%%%%%%%%%%%%%%%%%%%%%%%%%%%%%%%%

\textbf{Hypotheses.} In this section, we always assume $Y$ is a finite-dimensional Banach space with $\dimension Y \geq 1$, $n \in \mathscr{P}$, and $U$ is an open subset of $\mathbf{R}^n$.

%----------------------------------------------------------------------------------------------------------

The goal of this section is to prove an embedding theorem of $\mathscr{K}(X, \mathbf{E}(Y))$ into $\mathscr{K}(X \times Y, \homomorphism(Y, \mathbf{R}))$, see \ref{theorem: general homeomorphic embedding widetilde E}, and the embedding theorems for other test function spaces of interest follow immediately. For this purpose, we first present basic results about the function space $\mathscr{K}(X, Z)$.

\begin{lemma}
    \label{lemma: properties K(X, Z)}
    Suppose $X$ is a locally compact Hausdorff space and $W$ and $Z$ are locally convex Hausdorff spaces. Then, the following four statements hold.
    \begin{enumerate}
        \item
            \label{lemma: K(X, Z) pushforward}
            If $f: W \to Z$ is a continuous linear map, then post-composition with $f$ defines a continuous linear map $\mathscr{K}(X, W) \to \mathscr{K}(X, Z)$.
        \item 
            \label{lemma: K(X, Z) product}
            If $Z$ is the product of finitely many locally convex spaces $Z_1, Z_2, \dotsc, Z_n$, then we have the isomorphism of locally convex spaces
            \[
                \mathscr{K}(X, Z) \simeq \prod_{i=1}^n \mathscr{K}(X, Z_i);
            \]
            in particular, we have $\mathscr{K}(X, Z^n) \simeq \mathscr{K}(X, Z)^n$.
        \item 
            \label{lemma: norm topology and uniform convergence}
            If $Z$ is normed and $K$ is a compact subset of $X$, then the topology of $\mathscr{K}_K(X, Z)$ is given by the norm with value $\sup\im |f|$ at $f \in \mathscr{K}_K(X, Z)$.
        \item 
            \label{lemma: strict inductive limit K(X, Z)}
            If $K_i$ is a sequence of compact subsets of $X$ such that $K_i \subset \interior K_{i+1}$ for $i \in \mathscr{P}$ and $X = \bigcup_{i=1}^\infty K_i$, then $\mathscr{K}(X, Z)$ is the strict inductive limit of $\mathscr{K}_{K_i}(X, Z)$ as $i \to \infty$.
    \end{enumerate}
\end{lemma}
\begin{proof}
    It is straightforward to verify \eqref{lemma: K(X, Z) pushforward} from the definitions. The statements \eqref{lemma: K(X, Z) product} and \eqref{lemma: norm topology and uniform convergence} are proved in \cite[III, \S 1, No.\ 1]{Bourbaki_integration_MR2018901}. To prove \eqref{lemma: strict inductive limit K(X, Z)}, by \ref{remark: criterion strict inductive limit}, it is enough to check the inclusion map $\mathscr{K}_{K_i}(X, Z) \to \mathscr{K}_{K_j}(X, Z)$ is a homeomorphic embedding whenever $i \leq j$, and this is straightforward from the definitions.
\end{proof}

%----------------------------------------------------------------------------------------------------------

\begin{lemma}
    \label{lemma: separability of K_K(X, Z)}
    Suppose $X$ is a second-countable locally compact Hausdorff space, $K$ is a compact subset of $X$, and $Z$ is a separable Banach space.  Then, the function space $\mathscr{K}_K(X, Z)$ is a separable Banach space.
\end{lemma}
\begin{proof}
    Since $Z$ is a Banach space and $\mathscr{K}_K(X, Z)$ is endowed with the supremum norm, it is straightforward to check that $\mathscr{K}_K(X, Z)$ is a Banach space. Since $\mathscr{K}_K(X)$ is separable by \cite[2.2, 2.23]{Menne16a_MR3528825}, it follows from \ref{lemma: K(X) tensor Z density} that $\mathscr{K}_K(X, Z)$ is separable.
\end{proof}

%----------------------------------------------------------------------------------------------------------

\begin{lemma}
    \label{lemma: neighborhoods of 0 in K(X, Z)}
    Suppose $Z$ is a normed space, $X$ is a locally compact Hausdorff space, $K_i$ is a sequence of compact subsets of $X$ such that $K_0 = \varnothing$, $K_{i-1} \subset \interior K_i$ whenever $i \in \mathscr{P}$, and $X = \bigcup_{i=1}^\infty K_i$. Then, the subsets
    \[
        V_\alpha = \mathscr{K}(X, Z) \cap \{ f \with |f(x)| \leq \alpha(i)^{-1} \text{ for $i \in \mathscr{P}$ and $x \in X \without K_{i-1}$} \}
    \]
    corresponding to each $\alpha \in \mathscr{P}^\mathscr{P}$ form a fundamental system of neighborhoods of $0$ in $\mathscr{K}(X, Z)$.
\end{lemma}
\begin{proof}
    Clearly, $V_\alpha$ is a neighborhood of $0$ in $\mathscr{K}(X, Z)$ whenever $\alpha \in \mathscr{P}^\mathscr{P}$. Let $V$ be a convex neighborhood of $0$ in $\mathscr{K}(X, Z)$. There exists $\beta \in \mathscr{P}^\mathscr{P}$ such that 
    \[
        \mathbf{B}(0, \beta(i)^{-1}) \cap \mathscr{K}_{K_i}(X, Z) \subset V \cap \mathscr{K}_{K_i}(X, Z) \quad \text{whenever $i \in \mathscr{P}$}.
    \]
    Choose non-negative functions $\phi_i \in \mathscr{K}(X)$ for $i \in \mathscr{P}$ such that
    \[
        \spt \phi_i \subset (\interior K_{i+1}) \without K_{i-1} \quad \text{and} \quad \sum_{j=1}^\infty 2^{-j} \phi_j(x) = 1
    \]
    whenever $i \in \mathscr{P}$ and $x \in X$, and choose $\alpha \in \mathscr{P}^\mathscr{P}$ such that 
    \[
        \alpha(i)^{-1} \sup\im |\phi_i| \leq \beta(i+1)^{-1} \quad \text{whenever $i \in \mathscr{P}$}.
    \]
    Note that if $i, j \in \mathscr{P}$ and $f \in V_\alpha \cap \mathscr{K}_{K_i}(X, Z)$, then 
    \[
        \spt \phi_j f \subset (\interior K_{j+1}) \without K_{j-1} \quad \text{and} \quad \sup \im |\phi_j f| \leq \beta(j+1)^{-1};
    \]
    therefore
    \begin{gather*}
        \phi_j f \in V \cap \mathscr{K}_{K_{j+1}}(X, Z), \quad \text{if $j \leq i$}, \\
        \phi_j f = 0, \quad \text{if $j > i$}.
    \end{gather*}
    Since $V$ is convex, $0 \in V$, and $f = \sum_{j = 1}^i 2^{-j} (\phi_j f) + 2^{-i} \cdot 0$, we conclude that $f \in V$, hence $V_\alpha \subset V$.
\end{proof}

%----------------------------------------------------------------------------------------------------------

\begin{remark}
    The proof of \ref{lemma: neighborhoods of 0 in K(X, Z)} is adapted from the proof of \cite[p.\ 66, Th\'eor\`eme II]{Sch66_MR0209834}.
\end{remark}

%----------------------------------------------------------------------------------------------------------

Next, we will study the topological vector structure of $\mathbf{E}(Y)$.

%----------------------------------------------------------------------------------------------------------

\begin{lemma}
    \label{lemma: properties of E}
    The following three statements hold.
    \begin{enumerate}
        \item 
            \label{lemma: properties of E 1}
            The derivative defines a continuous monomorphism
            \[
                \mathbf{E}(Y) \to \mathscr{K}(Y, \homomorphism(Y, \mathbf{R}))
            \]
            and its restriction to $\mathbf{E}_s(Y)$ gives a norm-preserving embedding
            \[
                \mathbf{E}_s(Y) \to \mathscr{K}_{\mathbf{B}(0, s)}(Y, \homomorphism(Y, \mathbf{R}))
            \]
            whenever $0 \leq s < \infty$.
        \item 
            \label{lemma: properties of E 2}
            $\mathbf{E}_s(Y)$ is a separable Banach space whenever $0 \leq s < \infty$.
        \item 
            \label{lemma: properties of E 3}
            $\mathbf{E}(Y)$ is the strict inductive limit of $\mathbf{E}_i(Y)$ as $i \to \infty$.
    \end{enumerate}
\end{lemma}
\begin{proof}
    We first prove \eqref{lemma: properties of E 1}. From the definition, we see that the derivative defines a norm-preserving embedding $\mathbf{E}_s(Y) \to \mathscr{K}_{\mathbf{B}(0, s)}(Y, \homomorphism(Y, \mathbf{R}))$ whenever $0 \leq s < \infty$, hence a continuous monomorphism $\mathbf{E}(Y) \to \mathscr{K}(Y, \homomorphism(Y, \mathbf{R}))$ by the universal property \ref{remark: universal property locally convex final topology}. The second assertion then follows from \ref{lemma: properties K(X, Z)}\eqref{lemma: strict inductive limit K(X, Z)} and \ref{lemma: properties strict inductive limit}\eqref{lemma: properties strict inductive limit1}.

    Next, we aim to show \eqref{lemma: properties of E 2}. Let $f_i$ be a Cauchy sequence in $\mathbf{E}_s(Y)$. Then, $\der f_i$ is a Cauchy sequence in $\mathscr{K}_{\mathbf{B}(0, s)}(Y, \homomorphism(Y, \mathbf{R}))$; by \ref{remark: basic estimate E_s}, the sequence $f_i$ is also Cauchy with respect to the supremum metric. Therefore, their pointwise limits $f: Y \to \mathbf{R}$ and $F: Y \to \homomorphism(Y, \mathbf{R})$ are both continuous. By Taylor's theorem, we can show that $\der f$ exists and equals $F$. It follows that $\mathbf{E}_s(Y)$ is complete. The separability of $\mathbf{E}_s(Y)$ follows from \eqref{lemma: properties of E 1} and \ref{lemma: separability of K_K(X, Z)}.

    Finally, since the inclusion map $\mathbf{E}_r(Y) \to \mathbf{E}_s(Y)$ is a norm-preserving embedding whenever $0 \leq r \leq s < \infty$, \eqref{lemma: properties of E 3} follows from \ref{remark: criterion strict inductive limit}.
\end{proof}

%----------------------------------------------------------------------------------------------------------

\begin{remark}
    \label{remark: sujectivity continuous monomorphism E}
    If $Y = \mathbf{R}$, then both maps in \ref{lemma: properties of E}\eqref{lemma: properties of E 1} are isomorphisms of locally convex spaces since the formula
    \[
        \gamma(y) = \textstyle\int_0^1 \langle y, f(ty) \rangle \ud t \quad \text{whenever $y \in \mathbf{R}$}
    \]
    defines a member in $\mathbf{E}_s(\mathbf{R})$ whenever $f \in \mathscr{K}_{\mathbf{B}(0, s)}(\mathbf{R}, \homomorphism(\mathbf{R}, \mathbf{R}))$ and $0 \leq s < \infty$. If $Y = \mathbf{R}^k$ with $k \geq 2$, then both maps in \ref{lemma: properties of E}\eqref{lemma: properties of E 1} are not surjective; in fact, we may choose a non-negative function $\omega \in \mathscr{D}(\mathbf{R}, \mathbf{R})$ such that $\omega(r) = 1$ for $|r| \leq 1$ and define
    \[
        f(x) = \omega(|x)|) (x_2, -x_1, 0, \dotsc, 0) \quad \text{whenever $x = (x_1, x_2, \dots, x_k) \in \mathbf{R}^k$}.
    \]
    Then, there exists no $\gamma \in \mathbf{E}(\mathbf{R}^k)$ such that $\der \gamma = f$; here we identify $\mathbf{R}^k$ with $\homomorphism(\mathbf{R}^k, \mathbf{R})$.
\end{remark}

%----------------------------------------------------------------------------------------------------------

To show the map $\mathbf{E}(Y) \to \mathscr{K}(Y, \homomorphism(Y, \mathbf{R}))$ as in \ref{lemma: properties of E}\eqref{lemma: properties of E 1} is still a homeomorphic embedding for $\dimension Y \geq 2$, we need to study the fundamental system of neighborhoods of $0$ in $\mathbf{E}(Y)$; for this purpose, we introduce another function space $\widetilde{\mathbf{E}}(Y)$ that is isomorphic to $\mathbf{E}(Y)$ as locally convex spaces when $\dimension Y \geq 2$, but its members have compact support.

%----------------------------------------------------------------------------------------------------------

\begin{definition}
    \label{definition: function space widetilde E}
    We define
    \[
        \widetilde{\mathbf{E}}_s(Y) = \{ \widetilde{\gamma} \with \text{$\widetilde{\gamma}: Y \to \mathbf{R}$ is of class 1, $\spt \widetilde{\gamma} \subset \mathbf{B}(0, s)$} \} \quad \text{whenever $0 \leq s < \infty$}
    \]
    endowed with the norm $\widetilde{\gamma} \mapsto \sup \im \|\der \widetilde{\gamma}\|$, and the space
    \[
        \widetilde{\mathbf{E}}(Y) = \bigcup \{ \widetilde{\mathbf{E}}_s(Y) \with 0 \leq s < \infty \}
    \]
    is endowed with the locally convex final topology induced by the inclusion maps $\widetilde{\mathbf{E}}_s(Y) \to \widetilde{\mathbf{E}}(Y)$. 
\end{definition}

%----------------------------------------------------------------------------------------------------------

\begin{remark}
    \label{remark: properties of widetilde E}
    Similar arguments as in the proof of \ref{lemma: properties of E} show that $\widetilde{\mathbf{E}}(Y)$ is the strict inductive limit of the separable Banach spaces $\widetilde{\mathbf{E}}_i(Y)$ as $i \to \infty$.
\end{remark}

%----------------------------------------------------------------------------------------------------------

\begin{remark}
    \label{remark: relation E and widetilde E}
    The map $\widetilde{\gamma} \to \widetilde{\gamma} - \widetilde{\gamma}(0)$ defines an norm-preserving embedding from $\widetilde{\mathbf{E}}_s(Y)$ into $\mathbf{E}_s(Y)$ whenever $0 \leq s < \infty$, hence by \ref{remark: universal property locally convex final topology} a continuous monomorphism from $\widetilde{\mathbf{E}}(Y)$ into $\mathbf{E}(Y)$, and therefore by \ref{lemma: properties K(X, Z)}\eqref{lemma: K(X, Z) pushforward} post-composition defines a continuous monomorphism $I: \mathscr{K}(X, \widetilde{\mathbf{E}}(Y)) \to \mathscr{K}(X, \mathbf{E}(Y))$ whenever $X$ is a locally compact Hausdorff space. For $\dim Y \geq 2$, these maps are isomorphisms of locally convex spaces.
\end{remark}

%----------------------------------------------------------------------------------------------------------

\begin{lemma}
    \label{lemma: neighborhoods of 0 in widetilde E}
    Suppose $L_0 = \varnothing$ and $L_i = \mathbf{B}(0, i) \cap Y$ for $i \in \mathscr{P}$. Then, the subsets $\widetilde{W}_\alpha$ consisting of $\widetilde{\gamma} \in \widetilde{\mathbf{E}}(Y)$ satisfying
    \[
         |\widetilde{\gamma}(y)|+ \|\der \widetilde{\gamma}(y)\| \leq \alpha(i)^{-1} \quad \text{whenever $i \in \mathscr{P}$ and $y \in Y \without L_{i-1}$}
    \]
    corresponding to each $\alpha \in \mathscr{P}^\mathscr{P}$ form a fundamental system of neighborhoods of $0$ in $\widetilde{\mathbf{E}}(Y)$.
\end{lemma}
\begin{proof}
    Clearly, the set $\widetilde{W}_\alpha$ is convex, symmetric, and absorbent whenever $\alpha \in \mathscr{P}^\mathscr{P}$. Note that if 
    \begin{gather*}
        i \in \mathscr{P}, \quad \varepsilon = (2i)^{-1} \inf\{ \alpha(j)^{-1} \with j = 1, 2, \dotsc, i \} \\
        \widetilde{\gamma} \in \mathbf{B}(0, \varepsilon) \cap \widetilde{\mathbf{E}}_i(Y), \quad v \in Y, \quad |v| = 1,  \quad \text{and} \quad y(t) = (i-t)v,
    \end{gather*}
    then we have
    \begin{align*}
        |\widetilde{\gamma}(y(t))| 
        &\leq \textstyle\int_0^t |\langle -v, \der \widetilde{\gamma}(y(s))| \ud \mathscr{L}^1 \, s \\
        &\leq \textstyle\int_0^t \|\der \widetilde{\gamma}(y(s))\| \ud \mathscr{L}^1 \, s \\
        &\leq i \varepsilon 
    \end{align*}
    whenever $0 \leq t \leq i$. Thus, $\mathbf{B}(0, \varepsilon) \cap \widetilde{\mathbf{E}}_i(Y) \subset \widetilde{W}_\alpha \cap \widetilde{\mathbf{E}}_i(Y)$, and we conclude $\widetilde{W}_\alpha$ is a neighborhood of $0$ in $\widetilde{\mathbf{E}}(Y)$ whenever $\alpha \in \mathscr{P}^\mathscr{P}$.
    
    Let $V$ be a convex neighborhood of $0$ in $\widetilde{\mathbf{E}}(Y)$. Then, there exists $\beta \in \mathscr{P}^\mathscr{P}$ such that
    \[
        \mathbf{B}(0, \beta(i)^{-1}) \cap \widetilde{\mathbf{E}}_i(Y) \subset V \cap \widetilde{\mathbf{E}}_i(Y) \quad \text{whenever $i \in \mathscr{P}$}.
    \]
    Choose non-negative functions $\phi_i: Y \to \mathbf{R}$ of class $1$ for $i \in \mathscr{P}$ such that
    \[
        \spt \phi_i \subset (\interior L_{i+1}) \without L_{i-1} \quad \text{and} \quad \sum_{i=1}^\infty 2^{-i} \phi_i(y) = 1
    \]
    whenever $i \in \mathscr{P}$ and $y \in Y$, and choose $\alpha \in \mathscr{P}^\mathscr{P}$ such that 
    \[
        \alpha(i)^{-1} \sup\im (|\phi_i| + \|\der \phi_i\|) \leq \beta(i+1)^{-1} \quad \text{whenever $i \in \mathscr{P}$}.
    \]
    Note that if $i, j \in \mathscr{P}$ and $\widetilde{\gamma} \in \widetilde{W}_\alpha \cap \widetilde{\mathbf{E}}_i(Y)$, then 
    \[
        \spt \phi_j \widetilde{\gamma} \subset (\interior L_{j+1}) \without L_{j-1} \quad \text{and} \quad \sup \im \|\der (\phi_j \widetilde{\gamma})\| \leq \beta(j+1)^{-1};
    \]
    therefore
    \begin{gather*}
        \phi_j \widetilde{\gamma} \in V \cap \widetilde{\mathbf{E}}_{j+1}(Y), \quad \text{if $j \leq i$}, \\
        \phi_j \widetilde{\gamma} = 0, \quad \text{if $j > i$}.
    \end{gather*}
    Since $V$ is convex, $0 \in V$, and $\widetilde{\gamma} = \sum_{j = 1}^i 2^{-j} (\phi_j \widetilde{\gamma}) + 2^{-i} \cdot 0$, we conclude that $\widetilde{\gamma} \in V$, hence that $\widetilde{W}_\alpha \subset V$.
\end{proof}

%----------------------------------------------------------------------------------------------------------

\begin{remark}
    \label{remark: multiplication continuity}
    Suppose $X$ is a second-countable locally compact Hausdorff space. Whenever $\widetilde{W}_\alpha$ and $V_\alpha$ are as in \ref{lemma: neighborhoods of 0 in widetilde E} and \ref{lemma: neighborhoods of 0 in K(X, Z)}, respectively, the images of $\widetilde{W}_{\alpha} \times \widetilde{W}_{\alpha}$ and $V_\alpha \times V_\alpha$ under multiplication are contained in $\widetilde{W}_\alpha$ and $V_\alpha$, respectively. It follows from \cite[I, \S1, No.\ 6, Proposition 5]{Bourbaki_TVS_MR910295} that the maps $\widetilde{\mathbf{E}}(Y) \times \widetilde{\mathbf{E}}(Y) \to \widetilde{\mathbf{E}}(Y)$ and $\mathscr{K}(X) \times \mathscr{K}(X) \to \mathscr{K}(X)$ induced by multiplication are continuous.
\end{remark}

%----------------------------------------------------------------------------------------------------------

To prove the embedding theorem, we still need several lemmas.

%----------------------------------------------------------------------------------------------------------

\begin{lemma}
    \label{lemma: positive convex decreasing function}
    Suppose $\alpha \in \mathscr{P}^\mathscr{P}$. Then, there exists a positive convex decreasing function $h: \mathbf{R} \to \mathbf{R}$ of class $1$ such that
    \[
        h(x) < \alpha(i)^{-1} \quad \text{whenever $x \geq i-1$ and $i \in \mathscr{P}$}.
    \]
\end{lemma}
\begin{proof}
    Define $\beta$ and $\delta$ in $\mathscr{P}^\mathscr{P}$ by
    \begin{gather*}
        \beta(i) = \sup \{ \alpha(j) \with 1 \leq j \leq i \} + i; \\
        \delta(i) = 2^{\beta(i)}
    \end{gather*}
    whenever $i \in \mathscr{P}$. Note that $\delta$ satisfies
    \[
        \alpha(i) < \delta(i) \leq \delta(i+1)/2 \quad \text{and} \quad \delta(i)^{-1} - \delta(i+1)^{-1} > \delta(i+1)^{-1} - \delta(i+2)^{-1}
    \]
    whenever $i \in \mathscr{P}$. Then, the function $f: \mathbf{R} \to \mathbf{R}$ defined by
    \[
        f(x) = 
        \begin{cases}
            \dfrac{1}{\delta(1)} - \left( \dfrac{1}{\delta(1)} - \dfrac{1}{\delta(2)} \right)x & \text{if $x \leq 0$} \\
            \dfrac{1}{\delta(i)} - \left( \dfrac{1}{\delta(i)} - \dfrac{1}{\delta(i+1)} \right)(x - (i-1)) & \text{if $i \in \mathscr{P}$, $i-1 \leq x \leq i$}
        \end{cases}
    \]
    is a convex decreasing function such that
    \[
        f(x) < \alpha(i)^{-1} \quad \text{whenever $x \geq i-1$ and $i \in \mathscr{P}$}.
    \]
    Choose a function $g: \mathbf{R} \to \mathbf{R}$ of class $1$ such that $g \geq 0$, $\textstyle\int g \ud \mathscr{L}^1 = 1$, and $\spt g \subset \mathbf{R} \cap \{ x \with -1 \leq x \leq 0 \}$. 
    
    Let $h$ be the convolution $f*g$, and we will verify that $h$ has the desired properties. Clearly, $h$ is a positive function of class $1$. Since
    \[
        h(x) - f(x) = \textstyle\int_{\mathbf{R} \cap \{ y \with -1 \leq y \leq 0 \}} (f(x - y) - f(x))g(y) \ud \mathscr{L}^1 \, y < 0,
    \]
    it follows that $h \leq f$.
    To show that $h$ is decreasing, we compute
    \[
        h(x+t) - h(x) = \textstyle\int (f((x - y) + t) - f(x - y))g(y) \ud \mathscr{L}^1 \, y < 0
    \]
    whenever $x \in \mathbf{R}$ and $t > 0$.
    For $x, z \in \mathbf{R}$ and $0 \leq \lambda \leq 1$, we compute
    \begin{align*}
        (f*g)(\lambda x + (1 - \lambda) z)
        &= \textstyle\int f(\lambda x + (1 - \lambda) z - y) g(y) \ud \mathscr{L}^1 \, y \\
        &= \textstyle\int f(\lambda (x - y) + (1 - \lambda) (z - y)) g(y) \ud \mathscr{L}^1 \, y \\
        &\leq \lambda \textstyle\int f(x - y) g(y) \ud \mathscr{L}^1 \, y \\
        &\phantom{\leq} + (1 - \lambda) \textstyle\int f(z - y) g(y) \ud \mathscr{L}^1 \, y \\
        &= \lambda (f*g)(x) + (1 - \lambda) (f*g)(z)
    \end{align*}
    and it follows that $h$ is convex.
\end{proof}

%----------------------------------------------------------------------------------------------------------

\begin{lemma}
    \label{lemma: topology of uniform convergence and locally convex final topology}
    Suppose $X$ is a locally compact Hausdorff space, $K$ is a compact subset of $X$, and either $Y = \mathbf{R}$, $F_i = \mathscr{K}_{\mathbf{B}(0, i)}(\mathbf{R})$, and $F = \mathscr{K}(\mathbf{R})$, or $F_i = \widetilde{\mathbf{E}}_i(Y)$ and $F = \widetilde{\mathbf{E}}(Y)$. Then, the locally convex final topology $\mathcal{T}_1$ on $\mathscr{K}_K(X, F)$ induced by the inclusion maps $\mathscr{K}_K(X, F_i) \to \mathscr{K}_K(X, F)$ for $i \in \mathscr{P}$ is identical to the topology $\mathcal{T}_2$ of uniform convergence on $\mathscr{K}_K(X, F)$.
\end{lemma}
\begin{proof}
    Since the inclusion map
    \[
        \mathscr{K}_K(X, F_i) \to (\mathscr{K}_K(X, F), \mathcal{T}_2)
    \]
    is continuous whenever $i \in \mathscr{P}$ by \ref{lemma: properties K(X, Z)}\eqref{lemma: K(X, Z) pushforward}, it follows that $\mathcal{T}_2 \subset \mathcal{T}_1$. 
    
    Conversely, let $V$ be a convex neighborhood of $0$ in $(\mathscr{K}_K(X, F), \mathcal{T}_1)$. Then, there exists $\beta \in \mathscr{P}^\mathscr{P}$ such that
    \[
        \mathbf{B}(0, \beta(i)^{-1}) \cap \mathscr{K}_K(X, F_i) \subset V \cap \mathscr{K}_K(X, F_i) \quad \text{whenever $i \in \mathscr{P}$}.
    \]
    Let $L_0 = \varnothing$ and $L_{i} = \mathbf{B}(0, i) \cap Y$ for $i \in \mathscr{P}$. Choose non-negative functions $\phi_i: Y \to \mathbf{R}$ of class $1$ such that
    \[
        \spt \phi_i \subset (\interior L_{i+1}) \without L_{i-1} \quad \text{and} \quad \sum_{i=1}^\infty 2^{-i} \phi(y) = 1
    \]
    whenever $i \in \mathscr{P}$ and $y \in Y$, and choose $\alpha \in \mathscr{P}^\mathscr{P}$ such that
    \[
        \alpha(i)^{-1} \sup\im ( |\phi_i| + \|\der \phi_i\| ) \leq \beta(i+1)^{-1} \quad \text{whenever $i \in \mathscr{P}$}.
    \]
    Let $\eta \in \mathscr{K}_K(X, F)$ such that  $\im \eta \subset U_\alpha$, where $U_\alpha$ equals either $V_\alpha$ or $\widetilde{W}_\alpha$ as in \ref{lemma: neighborhoods of 0 in K(X, Z)} with $X$ and $Z$ both replaced by $\mathbf{R}$ and \ref{lemma: neighborhoods of 0 in widetilde E}, respectively. By \ref{lemma: properties K(X, Z)}\eqref{lemma: strict inductive limit K(X, Z)}, \ref{remark: properties of widetilde E}, and \ref{lemma: properties strict inductive limit}\eqref{lemma: properties strict inductive limit2}, there exists $i \in \mathscr{P}$ such that $\im \eta \subset F_i$. Since multiplication by $\phi_j$ induces a continuous linear map from $F$ into $F_{j+1}$ by \ref{remark: multiplication continuity}, \ref{lemma: properties K(X, Z)}\eqref{lemma: strict inductive limit K(X, Z)}, \ref{remark: properties of widetilde E}, and \ref{lemma: properties strict inductive limit}\eqref{lemma: properties strict inductive limit1} the function $\eta_j$ defined by $\eta_j(x) = \phi_j \eta(x)$ for $x \in X$ belongs to $\mathscr{K}_K(X, F_{j+1})$ by \ref{lemma: properties K(X, Z)}\eqref{lemma: K(X, Z) pushforward}. Observe that
    \[
        \sup \im |\eta_j| \leq \beta(j+1)^{-1}
    \]
    whenever $j \in \mathscr{P}$. Note that $\eta_j = 0$ for $j > i$. Therefore, we have
    \[
        \eta = \sum_{j=1}^i 2^{-j} \eta_j + 2^{-1} \cdot 0 \in V
    \]
    since $V$ is convex and $0 \in V$. It follows that $\mathcal{T}_1 \subset \mathcal{T}_2$.
\end{proof}

%----------------------------------------------------------------------------------------------------------

\begin{lemma}
    \label{lemma: double inductive limit K(X, E)}
    Suppose $X$ is a locally compact Hausdorff space, $K_i$ is a sequence of compact subsets of $X$ such that $X = \bigcup_{i=1}^\infty K_i$ and $K_i \subset \interior K_{i+1}$ whenever $i \in \mathscr{P}$, and either $Y = \mathbf{R}$, $F_i = \mathscr{K}_{\mathbf{B}(0, i)}(\mathbf{R})$, and $F = \mathscr{K}(\mathbf{R})$, or $F_i = \widetilde{\mathbf{E}}_i(Y)$ and $F = \widetilde{\mathbf{E}}(Y)$.

    Then, the following three statements hold.
    \begin{enumerate}
        \item 
            \label{lemma: double inductive limit K(X, E) 1}
            Whenever $K$ is a compact subset of $X$ and $j \in \mathscr{P}$, we have 
            \begin{gather*}
                \mathscr{K}_K(X, F) = \bigcup_{i=1}^\infty \mathscr{K}_K(X, F) \cap \{ f \with \im f \subset F_i \}, \\
                \mathscr{K}_K(X, F) \cap \{ f \with \im f \subset F_j \} = \mathscr{K}_K(X, F_j).
            \end{gather*}
            Consequently, $\mathscr{K}_K(X, F)$ is the strict inductive limit of $\mathscr{K}_K(X, F_i)$ as $i \to \infty$.
        \item 
            \label{lemma: double inductive limit K(X, E) 2}
            $\mathscr{K}(X, F)$ is the strict inductive limit of $\mathscr{K}_{K_i}(X, F_i)$ as $i \to \infty$. Moreover, this statement remains true with $F$ and $F_i$ replaced by $\mathbf{E}(Y)$ and $\mathbf{E}_i(Y)$, respectively.
        \item 
            \label{lemma: double inductive limit K(X, E) 3}
            The map $\iota: \mathscr{K}(X, \mathbf{E}(Y)) \to \mathscr{K}(X \times Y, \homomorphism(Y, \mathbf{R}))$ defined by
            \[
                \iota(\eta)(x, y) = (\der \eta(x))(y) \quad \text{for $\eta \in \mathscr{K}(X, \mathbf{E}(Y))$ and $(x, y) \in X \times Y$}
            \]
            is a continuous monomorphism; moreover, its restriction gives a norm-preserving embedding
            \[
                \mathscr{K}_K(X, \mathbf{E}_s(Y)) \to \mathscr{K}_{K \times \mathbf{B}(0, s)}(X \times Y, \homomorphism(Y, \mathbf{R}))
            \]
            whenever $K$ is a compact subset of $X$ and $0 \leq s < \infty$.
    \end{enumerate}
\end{lemma}
\begin{proof}
    The main assertion of \eqref{lemma: double inductive limit K(X, E) 1} follows from \ref{lemma: properties K(X, Z)}\eqref{lemma: strict inductive limit K(X, Z)}, \ref{remark: properties of widetilde E}, and \ref{lemma: properties strict inductive limit}\eqref{lemma: properties strict inductive limit1}\eqref{lemma: properties strict inductive limit2}. Therefore, the postscript of \eqref{lemma: double inductive limit K(X, E) 1} follows from \ref{remark: criterion strict inductive limit} and \ref{lemma: topology of uniform convergence and locally convex final topology}. 
    
    Next, we will prove \eqref{lemma: double inductive limit K(X, E) 2}. By the postscript of \eqref{lemma: double inductive limit K(X, E) 1} and \ref{lemma: double locally convex final topology}, we see the locally convex final topology on $\mathscr{K}(X, F)$ induced by the inclusion maps 
    \[
        \mathscr{K}_{K_i}(X, F) \to \mathscr{K}(X, F) \quad  \text{for $i \in \mathscr{P}$}
    \]
    is identical to the locally convex final topology induced by the inclusion maps 
    \[
        \mathscr{K}_{K_i}(X, F_j) \to \mathscr{K}(X, F) \quad \text{for $i, j \in \mathscr{P}$.}
    \]
    On the other hand, since the inclusion map $\mathscr{K}_{K_i}(X, F_j) \to \mathscr{K}_{K_k}(X, F_k)$ is a homeomorphic embedding whenever $i, j \in \mathscr{P}$ and $k = \sup\{ i, j \}$, the main assertion of \eqref{lemma: double inductive limit K(X, E) 2} follows from \ref{remark: criterion strict inductive limit}. To prove the postscript of \eqref{lemma: double inductive limit K(X, E) 2}, it is enough to note that for $\dimension Y = 1$ we have $\mathscr{K}_{\mathbf{B}(0, i)}(\mathbf{R}) \simeq \mathbf{E}_i(\mathbf{R})$ and $\mathscr{K}(\mathbf{R}) \simeq \mathbf{E}(\mathbf{R})$ by \ref{remark: sujectivity continuous monomorphism E}, and for $\dimension Y \geq 2$, we have $\widetilde{\mathbf{E}}_s(Y) \simeq \mathbf{E}_s(Y)$ and $\widetilde{\mathbf{E}}(Y) \simeq \mathbf{E}(Y)$ by \ref{remark: relation E and widetilde E}.

    Finally, we aim to prove \eqref{lemma: double inductive limit K(X, E) 3}. Whenever $K$ is a compact subset of $X$ and $0 \leq s < \infty$, by \ref{lemma: properties of E}\eqref{lemma: properties of E 1} and \ref{lemma: properties K(X, Z)}\eqref{lemma: K(X, Z) pushforward}, post-composition with the derivative induces a norm-preserving embedding
    \[
        \mathscr{K}_K(X, \mathbf{E}_s(Y)) \to \mathscr{K}_K(X, \mathscr{K}_{\mathbf{B}(0, s)}(Y, \homomorphism(Y, \mathbf{R}))),
    \]
    and it follows from \ref{lemma: exponential law} that so is
    \[
        \mathscr{K}_K(X, \mathbf{E}_s(Y)) \to \mathscr{K}_{K \times \mathbf{B}(0, s)}(X \times Y, \homomorphism(Y, \mathbf{R})).
    \]
    These embeddings then induce a continuous linear map that agrees with $\iota$ by \eqref{lemma: double inductive limit K(X, E) 2} and \ref{remark: universal property locally convex final topology}. The postscript follows from \ref{lemma: properties K(X, Z)}\eqref{lemma: strict inductive limit K(X, Z)} and\ref{lemma: properties strict inductive limit}\eqref{lemma: properties strict inductive limit1}.
\end{proof}

%----------------------------------------------------------------------------------------------------------

\begin{remark}
    \label{remark: relation iota and widetilde iota}
    If $\dimension Y = 1$, then the embeddings in \ref{lemma: double inductive limit K(X, E)}\eqref{lemma: double inductive limit K(X, E) 3} are norm-preserving isomorphisms and hence $\iota$ is an isomorphism of locally convex spaces. If $\dimension Y \geq 2$, then the map $I: \mathscr{K}(X, \widetilde{\mathbf{E}}(Y)) \to \mathscr{K}(X, \mathbf{E}(Y))$ as in \ref{remark: relation E and widetilde E} is an isomorphism; it follows that $\iota$ is a homeomorphic embedding if and only if so is the composition $\widetilde \iota = \iota \circ I$.
\end{remark}

The following theorem shows that the map $\widetilde{\iota}$ is a homeomorphic embedding (even for the case $\dimension Y = 1$); therefore, so is $\iota$.

%----------------------------------------------------------------------------------------------------------

\begin{theorem}
    \label{theorem: general homeomorphic embedding widetilde E}
    Suppose $X$ is a locally compact Hausdorff space, $K_0 = \varnothing$, $K_i$ is a sequence of compact subsets of $X$ such that $K_i \subset \interior K_{i+1}$ whenever $i \in \mathscr{P}$ and $\bigcup_{i=1}^\infty K_i = X$, let $L_0 = \varnothing$ and $L_i = \mathbf{B}(0, i) \cap Y$ whenever $i \in \mathscr{P}$, and let $C_i = K_i \times L_i$ whenever $i \in \mathscr{P} \cup \{ 0 \}$.
    
    Then, the subsets $\widetilde{W}_\alpha$ consisting of $\eta \in \mathscr{K}(X, \widetilde{\mathbf{E}}(Y))$ satisfying
    \[
        |\eta(x)(y)| + \| \der(\eta(x))(y) \| \leq \alpha(i)^{-1}
    \]
    whenever $i \in \mathscr{P}$ and $(x, y) \in (X \times Y) \without C_{i-1}$, corresponding to each $\alpha \in \mathscr{P}^\mathscr{P}$ form a fundamental system of neighborhoods of $0$ of $\mathscr{K}(X, \widetilde{\mathbf{E}}(Y))$. Furthermore, the map $\widetilde{\iota}: \mathscr{K}(X, \widetilde{\mathbf{E}}(Y)) \to \mathscr{K}(X \times Y, \homomorphism(Y, \mathbf{R}))$ defined by
    \[
        \widetilde{\iota}(\eta)(x, y) = \der (\eta(x))(y) \quad \text{for $(x, y) \in X \times Y$}
    \]
    is a homeomorphic embedding.
\end{theorem}
\begin{proof}
    Clearly, the set $\widetilde{W}_\alpha$ is convex and symmetric; by \ref{lemma: double inductive limit K(X, E)}\eqref{lemma: double inductive limit K(X, E) 1}, $\widetilde{W}_\alpha$ is absorbent whenever $\alpha \in \mathscr{P}^\mathscr{P}$. Note that if 
    \begin{gather*}
        i \in \mathscr{P}, \quad \varepsilon = (2i)^{-1} \inf\{ \alpha(j)^{-1} \with j = 1, 2, \dotsc, i \}, \\
        x \in X, \quad v \in Y, \quad |v| = 1,  \quad \text{and} \quad y(t) = (i-t)v, \\
        \eta \in \mathbf{B}(0, \varepsilon) \cap \mathscr{K}_{K_i}(X, \widetilde{\mathbf{E}}_i(Y))
    \end{gather*}
    then we have
    \begin{align*}
        |\eta(x)(y(t))| 
        &\leq \textstyle\int_0^t |\langle -v, \der (\eta(x))(y(s)) \rangle| \ud \mathscr{L}^1 \, s \\
        &\leq \textstyle\int_0^t \|\der (\eta(x))(y(s))\| \ud \mathscr{L}^1 \, s \\
        &\leq t \varepsilon
    \end{align*}
    whenever $0 \leq t \leq i$. Thus, $\mathbf{B}(0, \varepsilon) \cap \mathscr{K}_{K_i}(X, \widetilde{\mathbf{E}}_i(Y)) \subset \widetilde{W}_\alpha \cap \mathscr{K}_{K_i}(X, \widetilde{\mathbf{E}}_i(Y))$, and we conclude $\widetilde{W}_\alpha$ is a neighborhood of $0$ in $\mathscr{K}(X, \widetilde{\mathbf{E}}(Y))$ whenever $\alpha \in \mathscr{P}^\mathscr{P}$.
    
    Let $V$ be a convex neighborhood of $0$ in $\mathscr{K}(X, \widetilde{\mathbf{E}}(Y))$. Then, by \ref{lemma: double inductive limit K(X, E)}\eqref{lemma: double inductive limit K(X, E) 2}, there exists $\beta \in \mathscr{P}^\mathscr{P}$ such that
    \[
        \mathbf{B}(0, \beta(i)^{-1}) \cap \mathscr{K}_{K_i}(X, \widetilde{\mathbf{E}}_i(Y)) \subset V \cap \mathscr{K}_{K_i}(X, \widetilde{\mathbf{E}}_i(Y)) \quad \text{whenever $i \in \mathscr{P}$}.
    \]
    Choose non-negative continuous functions $f_i: X \to \mathbf{R}$ and non-negative functions $g_i: Y \to \mathbf{R}$ of class $1$ for $i \in \mathscr{P}$ such that
    \[
        \spt f_i \subset (\interior K_{i+1}) \without K_{i-1} \quad \text{and} \quad \spt g_i \subset (\interior L_{i+1}) \without L_{i-1}
    \]
    whenever $i \in \mathscr{P}$ and such that
    \[
        \sum_{i=1}^\infty f_i(x) = 1 \quad \text{and} \quad \sum_{i=1}^\infty g_i(y) = 1
    \]
    whenever $(x, y) \in X \times Y$.
    For $i \in \mathscr{P}$, we define a continuous function $\phi_i$ with values in $\widetilde{\mathbf{E}}(Y)$ by
    \[
        \phi_i(x) = 2^{i} \left(\left( \sum_{j = 1}^i f_j(x) \right)\left( \sum_{j = 1}^i g_j \right) - \left( \sum_{j = 1}^{i-1} f_j(x) \right)\left( \sum_{j = 1}^{i-1} g_j \right) \right) \quad \text{for $x \in X$}.
    \]
    Note that
    \[
        (X \times Y) \cap \{ (x, y) \with \phi_i(x)(y) \neq 0 \} \subset (\interior C_{i+1}) \without C_{i-1} \quad \text{for $i \in \mathscr{P}$}
    \]
    and that
    \[
        \quad \sum_{i=1}^\infty 2^{-i} \phi_i(x)(y) = 1 \quad \text{for $(x, y) \in X \times Y$}.
    \]
    Choose $\alpha \in \mathscr{P}^\mathscr{P}$ such that 
    \[
        \alpha(i)^{-1} (|\phi_i(x)(y)| + \|\der (\phi_i(x))(y)\|) \leq \beta(i+1)^{-1}
    \]
    whenever $(x, y) \in X \times Y$ and $i \in \mathscr{P}$. If $i, j \in \mathscr{P}$ and $\eta \in\widetilde{W}_\alpha \cap \mathscr{K}_{K_i}(X, \widetilde{\mathbf{E}}_i(Y))$, then the function $\phi_j \eta$ defined by
    \[
        (\phi_j \eta)(x) = \phi_j(x)\eta(x) \quad \text{whenever $x \in X$}
    \]
    satisfies $\phi_j \eta \in \mathscr{K}_{K_{j+1}}(X, \widetilde{\mathbf{E}}_{j+1}(Y))$ because of \ref{remark: multiplication continuity} and \ref{remark: properties of widetilde E}, and
    \[
        \sup \{ \sup\im \|\der ((\phi_j \eta)(x))\| \with x \in X \} \leq \beta(j+1)^{-1};
    \]
    therefore
    \begin{gather*}
        \phi_j \eta \in V \cap \mathscr{K}_{K_{j+1}}(X, \widetilde{\mathbf{E}}_{j+1}(Y)) \quad \text{if $j \leq i$}, \\
        \phi_j \eta = 0 \quad \text{if $j > i$}.
    \end{gather*}
    Since $V$ is convex, $0 \in V$, and $\eta = \sum_{j = 1}^i 2^{-j} (\phi_j \eta) + 2^{-i} \cdot 0$, we conclude that $\eta \in V$, hence $\widetilde{W}_\alpha \subset V$.
    
    Finally, by \ref{remark: relation iota and widetilde iota}, we see $\widetilde{\iota}$ is a continuous monomorphism; hence, it remains to prove that $\widetilde{\iota}^{-1}$ is continuous. Let $\alpha \in \mathscr{P}^\mathscr{P}$ and let $\widetilde{W}_\alpha$ be a neighborhood of $0$ in $\mathscr{K}(X, \widetilde{\mathbf{E}}(Y))$ defined as above. Now, we aim to find $\delta \in \mathscr{P}^\mathscr{P}$ such that 
    \[
        V_\delta \cap \im \widetilde{\iota} \subset \widetilde{\iota}[\widetilde{W}_\alpha]
    \]
    where $V_\delta$ is a neighborhood of $0$ in $\mathscr{K}(X \times Y, \homomorphism(Y, \mathbf{R}))$ as in \ref{lemma: neighborhoods of 0 in K(X, Z)} with $X$, $K_i$, and $Z$ replaced by $X \times Y$, $C_i$, and $\homomorphism(Y, \mathbf{R})$, respectively. By \ref{lemma: positive convex decreasing function}, there exists a positive convex decreasing function $h: \mathbf{R} \to \mathbf{R}$ of class $1$ such that $h(i-1) \leq 2^{-1}\alpha(i)^{-1}$ whenever $i \in \mathscr{P}$. Then, we define $\delta \in \mathscr{P}^\mathscr{P}$ to satisfy
    \[
        \delta(i)^{-1} \leq \inf \{ -h'(i) , 2^{-1}\alpha(i)^{-1} \} \quad \text{whenever $i \in \mathscr{P}$}
    \]
    and let $V_\delta$ be as in \ref{lemma: neighborhoods of 0 in K(X, Z)}. If 
    \begin{gather*}
        i \in \mathscr{P}, \quad \eta \in \widetilde{\iota}^{-1} [V_\delta] \cap \mathscr{K}_{K_i}(X, \widetilde{\mathbf{E}}_i(Y)), \\
        v \in Y, \quad |v| = 1,  \quad \text{and} \quad y(t) = (i-t)v,
    \end{gather*}
    then we have
    \begin{align*}
        |\eta(x)(y(t))| 
        &\leq \textstyle\int_0^t |\langle -v, \der (\eta(x)) (y(s))| \ud \mathscr{L}^1 \, s \\
        &\leq \textstyle\int_0^t \|\der (\eta(x)) (y(s))\| \ud \mathscr{L}^1 \, s \\
        &\leq \textstyle\int_0^t -h'(i-s) \ud \mathscr{L}^1 \, s \\
        &\leq h(|y(t)|)
    \end{align*}
    whenever $x \in X$ and $0 \leq t \leq i$. It follows that $|\eta(x)(y)| \leq h(|y|)$ whenever $(x, y) \in X \times Y$, hence $\widetilde{\iota}^{-1} [V_\delta] \cap \mathscr{K}_{K_i}(X, \widetilde{\mathbf{E}}_i(Y)) \subset \widetilde{W}_\alpha$.
\end{proof}

%----------------------------------------------------------------------------------------------------------

\begin{corollary}
    \label{corollary: general homeomorphic embedding E}
    Suppose $X$ is a locally compact Hausdorff space, $K_0 = \varnothing$, $K_i$ is a sequence of compact subsets of $X$ such that $K_i \subset \interior K_{i+1}$ whenever $i \in \mathscr{P}$ and $\bigcup_{i=1}^\infty K_i = X$, and let $C_i = K_i \times (\mathbf{B}(0, i) \cap Y)$ whenever $i \in \mathscr{P} \cup \{ 0 \}$. 
    
    Then, the map $\iota: \mathscr{K}(X, \mathbf{E}(Y)) \to \mathscr{K}(X \times Y, \homomorphism(Y, \mathbf{R}))$ defined by
    \[
        \iota(\eta)(x, y) = \der (\eta(x))(y) \quad \text{for $(x, y) \in X \times Y$}
    \]
    is a homeomorphic embedding. Furthermore, the subsets $W_\alpha$ consisting of $\eta \in \mathscr{K}(X, \mathbf{E}(Y))$ satisfying
    \[
        \| \der(\eta(x))(y) \| \leq \alpha(i)^{-1} \quad \text{whenever $i \in \mathscr{P}$ and $(x, y) \in (X \times Y) \without C_{i-1}$}
    \]
    corresponding to each $\alpha \in \mathscr{P}^\mathscr{P}$ form a fundamental system of neighborhoods of $0$ of $\mathscr{K}(X, \mathbf{E}(Y))$.
\end{corollary}
\begin{proof}
    It follows from \ref{remark: relation iota and widetilde iota}, \ref{theorem: general homeomorphic embedding widetilde E}, and \ref{lemma: neighborhoods of 0 in K(X, Z)}.
\end{proof}

%----------------------------------------------------------------------------------------------------------

\begin{corollary}
    \label{corollary: homeomorphic embedding E}
    The map $\mathbf{E}(Y) \to \mathscr{K}(Y, \homomorphism(Y, \mathbf{R}))$ defined as in \ref{lemma: properties of E}\eqref{lemma: properties of E 1} is a homeomorphic embedding.
    
    Furthermore, the subsets $W_\alpha$ consisting of $\gamma \in \mathbf{E}(Y)$ satisfying
    \[
        \| \der \gamma(y) \| \leq \alpha(i)^{-1} \quad \text{whenever $i \in \mathscr{P}$ and $|y| > i-1$}
    \]
    corresponding to each $\alpha \in \mathscr{P}^\mathscr{P}$ form a fundamental system of neighborhoods of $0$ of $\mathbf{E}(Y)$.
\end{corollary}
\begin{proof}
    If $\dimension Y = 1$, it is proved in \ref{remark: sujectivity continuous monomorphism E}, and the postscript follows from \ref{lemma: neighborhoods of 0 in K(X, Z)}. If $\dimension Y \geq 2$, it follows from \ref{corollary: general homeomorphic embedding E} applied with $X$ and $K_i = X$ for $i \in \mathscr{P}$ being a singleton.
\end{proof}

%----------------------------------------------------------------------------------------------------------

\begin{corollary}
    \label{corollary: dual map epimorphism}
    The dual map $\mathscr{K}(Y, \homomorphism(Y, \mathbf{R}))^* \to \mathbf{E}(Y)^*$ is an epimorphism.
\end{corollary}
\begin{proof}
    It follows from \ref{corollary: homeomorphic embedding E} and the Hahn-Banach theorem \cite[II, \S 4, No.\ 1,  Proposition 2]{Bourbaki_TVS_MR910295}.
\end{proof}

%----------------------------------------------------------------------------------------------------------

\begin{remark}
    \label{remark: non-unique representation}
    The following example shows that the dual map in \ref{corollary: dual map epimorphism} is not injective for $\dimension Y = k > 1$. Assume $Y = \mathbf{R}^k$. Whenever $\omega \in \mathscr{D}(\mathbf{R}, \mathbf{R})$, $0 \notin \spt \omega'$, and $L$ is a non-zero anti-symmetric endomorphism on $\mathbf{R}^k$, the functional $\mu \in \mathscr{K}(\mathbf{R}^k, \homomorphism(\mathbf{R}^k, \mathbf{R}))^*$ defined by
    \[
        \mu(\delta) = \textstyle\int \langle X(y), \delta(y) \rangle \ud \mathscr{L}^k \, y \quad \text{for $\delta \in \mathscr{K}(\mathbf{R}^k, \homomorphism(\mathbf{R}^k, \mathbf{R}))$}
    \]
    belongs to the kernel of the dual map $\mathscr{K}(\mathbf{R}^k, \homomorphism(\mathbf{R}^k, \mathbf{R}))^* \to \mathbf{E}(\mathbf{R}^k)^*$, where $X \in \mathscr{D}(\mathbf{R}^k, \mathbf{R}^k)$ is defined by
    \[
        X(y) = \omega(|y|) L(y) \quad \text{for $y \in \mathbf{R}^k$},
    \]
    because $\divergence X = 0$.
\end{remark}

%----------------------------------------------------------------------------------------------------------

To unify the test function spaces from domain and codomain, we also introduce the test function space on the product.

%----------------------------------------------------------------------------------------------------------

\begin{definition}
    \label{definition: function space H}
    Whenever $C$ is a compact subset of $U \times Y$, the function space $\mathbf{H}_C(U \times Y, \mathbf{R}^n)$ consists of all functions $\eta: U \times Y \to \mathbf{R}^n$ such that
    \begin{enumerate}
        \item 
        \label{definition: function space H 1}
            $\eta(x, \cdot) \bullet v \in \mathbf{E}(Y)$ whenever $x \in U$ and $v \in \mathbf{R}^n$.
        \item
        \label{definition: function space H 2}
            The function
            \[
                (x, y) \mapsto \der (\eta(x, \cdot)) (y) \in \homomorphism(Y, \mathbf{R}^n)
            \]
            is continuous with compact support in $C$.
    \end{enumerate}
    We endow $\mathbf{H}_C(U \times Y, \mathbf{R}^n)$ with the norm (see \ref{remark: H is a normed space}) whose value equal
    \[
        \sup \{ \|\der(\eta(x, \cdot))(y)\| \with (x, y) \in U \times Y \} \quad \text{at $\eta \in \mathbf{H}_C(U \times Y, \mathbf{R}^n)$},
    \]
    and endow 
    \[
        \mathbf{H}(U \times Y, \mathbf{R}^n) = \bigcup \{ \mathbf{H}_C(U \times Y, \mathbf{R}^n) \with \text{$C$ is a compact subset of $U \times Y$} \}
    \]
    with the locally convex final topology induced by the inclusion maps 
    \[
        \mathbf{H}_C(U \times Y, \mathbf{R}^n) \to \mathbf{H}(U \times Y, \mathbf{R}^n).
    \]
\end{definition}

%----------------------------------------------------------------------------------------------------------

\begin{remark}
    \label{remark: H is a normed space}
    We shall verify that $\mathbf{H}_C(U \times Y, \mathbf{R}^n)$ is a normed space. Let $\eta \in \mathbf{H}_C(U \times Y, \mathbf{R}^n)$. Note that \ref{definition: function space H}\eqref{definition: function space H 1} implies that $\eta(x, 0) = 0$ whenever $x \in U$. If $(x, y) \in U \times Y$, then we have
    \[
        \eta(x, y) = \textstyle\int_0^1 \langle y, \der (\eta(x, \cdot))(t y) \rangle \ud \mathscr{L}^1 \, t;
    \]
    in particular, $\eta$ is continuous. If $\eta$ satisfies
    \[
        \sup \{ \|\der(\eta(x, \cdot))(y)\| \with x \in U \text{ and } y\in Y \} = 0,
    \]
    then $\eta = 0$. Thus, $\mathbf{H}_C(U \times Y, \mathbf{R}^n)$ is a normed space.
\end{remark}

%----------------------------------------------------------------------------------------------------------

\begin{lemma}
    \label{lemma: properties of H}
    The following two statements hold.
    \begin{enumerate}
        \item 
        \label{lemma: properties of H 1}
            Whenever $K$ is a compact subset of $U$ and $0 \leq s < \infty$, we have the following isomorphisms of Banach spaces
            \[
                \mathbf{H}_{K \times \mathbf{B}(0, s)}(U \times Y, \mathbf{R}^n) \simeq \mathscr{K}_K(U, \mathbf{E}_s(Y))^n;
            \]
            in particular, $\mathbf{H}_{K \times \mathbf{B}(0, s)}(U \times Y, \mathbf{R}^n)$ is a separable Banach space.
        \item 
        \label{lemma: properties of H 2}
            Whenever $C_i$ is a sequence of compact subsets of $U \times Y$ such that $U \times Y = \bigcup_{i = 1}^\infty C_i$ and $C_i \subset \interior C_{i+1}$ for $i \in \mathscr{P}$, $\mathbf{H}(U \times Y, \mathbf{R}^n)$ is the strict inductive limit of $\mathbf{H}_{C_i}(U \times Y, \mathbf{R}^n)$ as $i \to \infty$.
    \end{enumerate}
\end{lemma}
\begin{proof}
    To prove \eqref{lemma: properties of H 1}, let $K$ be a compact subset of $U$ and $0 \leq s < \infty$. From \ref{definition: function space H}\eqref{definition: function space H 2}, the map 
    \[
        \mathbf{H}_{K \times \mathbf{B}(0, s)}(U \times Y, \mathbf{R}^n) \to \mathscr{K}_{K \times \mathbf{B}(0, s)}(U \times Y, \homomorphism(Y, \mathbf{R}^n))
    \]
    defined by
    \[
        \eta \mapsto [(x, y) \mapsto \der(\eta(x, \cdot))(y)]
    \]
    is a homeomorphic embedding; noting from \ref{lemma: properties K(X, Z)}\eqref{lemma: K(X, Z) product} that  
    \[
        \mathscr{K}_{K \times \mathbf{B}(0, s)}(U \times Y, \homomorphism(Y, \mathbf{R}^n)) \simeq \mathscr{K}_{K \times \mathbf{B}(0, s)}(U \times Y, \homomorphism(Y, \mathbf{R}))^n,
    \]
    so is the map 
    \[
        f: \mathbf{H}_{K \times \mathbf{B}(0, s)}(U \times Y, \mathbf{R}^n) \to \mathscr{K}_{K \times \mathbf{B}(0, s)}(U \times Y, \homomorphism(Y, \mathbf{R}))^n.
    \]
    On the other hand, recall from \ref{lemma: double inductive limit K(X, E)}\eqref{lemma: double inductive limit K(X, E) 3} that the map 
    \[
        J: \mathscr{K}_K(U, \mathbf{E}_s(Y)) \to \mathscr{K}_{K \times \mathbf{B}(0, s)}(U \times Y, \homomorphism(Y, \mathbf{R}))
    \]
    defined by
    \[
        J(\eta)(x, y) = \der(\eta(x))(y) \quad \text{whenever $(x, y) \in U \times Y$}
    \]
    is a homeomorphic embedding, and so is its Cartesian product $g = \prod_{i = 1}^n J$. Since $\im f = \im g$, it follows that $g^{-1} \circ f$ defines an isomorphism between the normed spaces $\mathbf{H}_{K \times \mathbf{B}(0, s)}(U \times Y, \mathbf{R}^n)$ and $\mathscr{K}_K(U, \mathbf{E}_s(Y))^n$; in particular, $\mathbf{H}_{K \times \mathbf{B}(0, s)}(U \times Y, \mathbf{R}^n)$ is a separable Banach space by \ref{lemma: properties of E}\eqref{lemma: properties of E 2} and \ref{lemma: separability of K_K(X, Z)}.

    Since the inclusion map $\mathbf{H}_{C_i}(U \times Y, \mathbf{R}^n) \to \mathbf{H}_{C_j}(U \times Y, \mathbf{R}^n)$ is a norm preserving embedding whenever $i \leq j \in \mathscr{P}$, 
    \eqref{lemma: properties of H 2} follows from \ref{remark: criterion strict inductive limit}.
\end{proof}

%----------------------------------------------------------------------------------------------------------

\begin{remark}
    \label{remark: isomorphism H}
    Suppose $K_i$ is a sequence of compact subsets of $U$ such that $U = \bigcup_{i=1}^\infty K_i$ and $K_i \subset \interior K_{i+1}$ for $i \in \mathscr{P}$. By \ref{lemma: properties K(X, Z)}\eqref{lemma: K(X, Z) pushforward}\eqref{lemma: K(X, Z) product}, \ref{theorem: inductive limit and finite product}, and \ref{lemma: double inductive limit K(X, E)}\eqref{lemma: double inductive limit K(X, E) 2}, we see $\mathscr{K}(U, \mathbf{E}(Y))^n \simeq \mathscr{K}(U, \mathbf{E}(Y)^n) \simeq \mathscr{K}(U, \mathbf{R}^n \otimes \mathbf{E}(Y))$ is the strict inductive limit of $\mathscr{K}_{K_i}(U, \mathbf{E}_i(Y))^n \simeq \mathscr{K}_{K_i}(U, \mathbf{E}_i(Y)^n) \simeq \mathscr{K}_{K_i}(U, \mathbf{R}^n \otimes \mathbf{E}_i(Y))$ as $i \to \infty$. Moreover, we have 
    \[
        \mathbf{H}(U \times Y, \mathbf{R}^n) \simeq \mathscr{K}(U, \mathbf{E}(Y))^n \simeq \mathscr{K}(U, \mathbf{E}(Y)^n) \simeq \mathscr{K}(U, \mathbf{R}^n \otimes \mathbf{E}(Y))
    \]
    by \ref{lemma: properties of H}\eqref{lemma: properties of H 1}\eqref{lemma: properties of H 2}.
\end{remark}

%----------------------------------------------------------------------------------------------------------

\begin{theorem}
    \label{theorem: homeomorphic embedding of H}
    The map $\iota: \mathbf{H}(U \times Y, \mathbf{R}^n) \to \mathscr{K}(U \times Y, \homomorphism(Y, \mathbf{R^n}))$ defined by
    \[
        \iota(\eta) = (\der \eta(x, \cdot))(y) \quad \text{whenever $\eta \in \mathbf{H}(U \times Y, \mathbf{R}^n)$ and $(x, y) \in U \times Y$}
    \]
    is a homeomorphic embedding. 
    
    Furthermore, if $K_0 = \varnothing$ and $K_i$ is a sequence of compact subsets of $U$ such that $U = \bigcup_{i=1}^\infty K_i$ and $K_i \subset \interior K_{i+1}$ whenever $i \in \mathscr{P}$, then the subsets $W_\alpha$ consisting of $\eta \in \mathbf{H}(U \times Y, \mathbf{R}^n)$ satisfying
    \[
        \sup\im \|(\der \eta(x, \cdot))(y)\| \leq \alpha(i)^{-1}
    \]
    whenever $i \in \mathscr{P}$ and $(x, y) \in (U \times Y) \without (K_{i-1} \times \mathbf{B}(0, i-1))$, corresponding to each $\alpha \in \mathscr{P}^\mathscr{P}$ form a fundamental system of neighborhoods of $0$ of $\mathbf{H}(U \times Y, \mathbf{R}^n)$.
\end{theorem}
\begin{proof}
    By \ref{remark: isomorphism H}, we have $\mathbf{H}(U \times Y, \mathbf{R}^n) \simeq \mathscr{K}(U, \mathbf{E}(Y))^n$. Since, by \ref{corollary: general homeomorphic embedding E}, \ref{theorem: inductive limit and finite product}, and \ref{lemma: properties K(X, Z)}\eqref{lemma: K(X, Z) product}, the composition
    \[
        \mathscr{K}(U, \mathbf{E}(Y))^n \to \mathscr{K}(U \times Y, \homomorphism(Y, \mathbf{R}))^n \simeq \mathscr{K}(U \times Y, \homomorphism(Y, \mathbf{R}^n))
    \]
    is a homeomorphic embedding and the main assertion follows. The postscript follows from \ref{lemma: neighborhoods of 0 in K(X, Z)}.
\end{proof}

%----------------------------------------------------------------------------------------------------------

\begin{remark}
    \label{remark: homeomorphic isomorphism H when Y = R}
    In case $Y = \mathbf{R}$, the homeomorphic embedding 
    \[
        \mathbf{H}(U \times \mathbf{R}, \mathbf{R}^n) \to \mathscr{K}(U \times \mathbf{R}, \homomorphism(\mathbf{R}, \mathbf{R}^n)) \simeq \mathscr{K}(U \times \mathbf{R}, \mathbf{R}^n)
    \]
    is surjective; in fact, for $\phi \in \mathscr{K}(U \times \mathbf{R}, \mathbf{R}^n)$, the function $\eta: U \times \mathbf{R} \to \mathbf{R}^n$ defined by
    \[
        \eta(x, y) = \textstyle\int_0^y \phi(x, z) \ud \mathscr{L}^1 \, z
    \]
    is a member in $\mathbf{H}(U \times \mathbf{R}, \mathbf{R}^n)$ whose image under the embedding equals $\phi$.
\end{remark}

%----------------------------------------------------------------------------------------------------------

\begin{lemma}
    \label{lemma: inductive limit, dense subset, extension}
    Suppose $A$ is a directed set, $(F_\alpha, f_{\beta \alpha})$ is an inductive system of locally convex spaces relative to $A$, $(D_\alpha, \delta_{\beta \alpha})$ is an inductive system of vector spaces relative to $A$, we denote by $(F, f_\alpha)$ and $(D, \delta_\alpha)$ the inductive limit of $(F_\alpha, f_{\beta \alpha})$ and $(D_\alpha, \delta_{\beta \alpha})$ respectively, for $\alpha \in A$, $D_\alpha$ is a dense subspace of $F_\alpha$ and the inclusion map $\iota_\alpha: D_\alpha \to F_\alpha$ satisfies
    \[
        f_{\beta \alpha} \circ \iota_\alpha = \iota_\beta \circ \delta_{\beta \alpha} \quad \text{whenever $\alpha \leq \beta \in A$}.
    \]
    Then, the vector space $D = \bigcup_{\alpha \in A} \delta_\alpha[D_\alpha]$ can be identified as a dense subspace of $F$. Moreover, if $D_\alpha$ is endowed with the subspace topology induced by $\iota_\alpha$, $G$ is a complete Hausdorff locally convex space, and the family of continuous linear maps $g_\alpha: D_\alpha \to G$ for $\alpha \in A$ satisfies
    \[
        g_\alpha = g_\beta \circ \delta_{\beta \alpha} \quad \text{whenever $\alpha \leq \beta \in A$},
    \]
    then the inductive limit of the unique extensions of $g_\alpha$ equals the unique extension of the inductive limit of $g_\alpha$; here, by convention, we also view $G$ as the limit of the constant inductive system $(G, \mathbf{1}_G)$ of locally convex spaces relative to $A$.
\end{lemma}
\begin{proof}
    From \ref{remark: inductive limit of maps}, the inductive limit $\iota$ of the inclusion maps $\iota_\alpha: D_\alpha \to F_\alpha$ is a monomorphism, and we identify $D$ as its image under $\iota$ in $F$. Since $D_\alpha$ is dense in $F_\alpha$, we have
    \[
        \im f_\alpha \subset \closure \im (f_\alpha \circ \iota_\alpha) = \closure \im (\iota \circ \delta_\alpha) \quad \text{whenever $\alpha \in A$}.
    \]
    It follows that
    \begin{align*}
        F = \bigcup\{ \im f_\alpha \with \alpha \in A \} 
        &\subset \bigcup\{ \closure \im (\iota \circ \delta_\alpha) \with \alpha \in A \} \\
        &\subset \closure \bigcup\{ \im (\iota \circ \delta_\alpha) \with \alpha \in A \} \\
        &= \closure \iota[D].
    \end{align*}
    To prove the postscript, we denote by $h_\alpha$ the unique continuous extension of $g_\alpha$ for $\alpha \in A$, and these $h_\alpha$ satisfy
    \[
        h_\alpha = h_\beta \circ f_{\beta \alpha} \quad \text{whenever $\alpha \leq \beta \in A$}.
    \]
    Then, the inductive limit $h$ of $h_\alpha$ exists and satisfies 
    \[
        h \circ (\iota \circ \delta_\alpha) = h \circ (f_\alpha \circ \iota_\alpha) = g_\alpha \quad \text{whenever $\alpha \in A$}.
    \]
    It follows that $h \circ \iota$ equals the inductive limit of $g_\alpha$. Finally, if $D$ is endowed with the subspace topology, then $h$ is the unique continuous extension of the inductive limit of $g_\alpha$.
\end{proof}

%----------------------------------------------------------------------------------------------------------

\begin{theorem}
    \label{theorem: density H}
    Suppose $\mu: \mathscr{D}(U, \mathbf{R}^n) \otimes (\mathscr{E}(Y, \mathbf{R}) \cap \mathbf{E}(Y)) \to \mathbf{H}(U \times Y, \mathbf{R}^n)$ is the canonical map defined by
    \[
        \mu(\theta \otimes \gamma)(x, y) = \gamma(y)\theta(x)
    \]
    for $\theta \in \mathscr{D}(U, \mathbf{R}^n)$, $\gamma \in \mathscr{E}(Y, \mathbf{R}) \cap \mathbf{E}(Y)$, and $(x, y) \in U \times Y$. Then, its image $\im \mu$ is dense in $\mathbf{H}(U \times Y, \mathbf{R}^n)$.
\end{theorem}
\begin{proof}
    Let $K$ be a compact subset of $U$ and $0 \leq s < \infty$. By \ref{lemma: K(X) tensor Z density}, the image of the canonical monomorphism
    \[
        \mathscr{K}_K(U) \otimes \mathbf{E}_s(Y) \to \mathscr{K}_K(U, \mathbf{E}_s(Y))
    \]
    is dense; therefore, the image of its Cartesian product
    \[
        (\mathscr{K}_K(U) \otimes \mathbf{E}_s(Y))^n \to \mathscr{K}_K(U, \mathbf{E}_s(Y))^n
    \]
    is also dense. On the other hand, observe that 
    \[
        \mathscr{K}_K(U, \mathbf{R}^n) \otimes \mathbf{E}_s(Y) \simeq (\mathscr{K}_K(U) \otimes \mathbf{E}_s(Y))^n,
    \]
    and by \ref{lemma: properties of H}\eqref{lemma: properties of H 1} we have $\mathscr{K}_K(U, \mathbf{E}_s(Y))^n \simeq \mathbf{H}_{K \times \mathbf{B}(0, s)}(U \times Y, \mathbf{R}^n)$ as Banach spaces. Thus, the composition
    \[
        \mathscr{K}_K(U, \mathbf{R}^n) \otimes \mathbf{E}_s(Y) \to \mathbf{H}_{K \times \mathbf{B}(0, s)}(U \times Y, \mathbf{R}^n)
    \]
    of the previous maps is a monomorphism whose image is dense. By convolution, we see $\mathscr{D}_K(U, \mathbf{R}^n)$ and $\mathscr{E}(Y, \mathbf{R}) \cap \mathbf{E}_s(Y)$ are dense in $\mathscr{K}_K(U, \mathbf{R}^n)$ and $\mathbf{E}_s(Y)$, respectively. Therefore, we may identify $\mathscr{D}_K(U, \mathbf{R}^n) \otimes (\mathscr{E}(Y, \mathbf{R}) \cap \mathbf{E}_s(Y))$ as a dense subspace of $\mathbf{H}_{K \times \mathbf{B}(0, s)}(U \times Y, \mathbf{R}^n)$.

    Let $K_i$ be a sequence of compact subset of $U$ such that $U = \bigcup_{i = 1}^\infty K_i$ and $K_i \subset \interior K_{i + 1}$ whenever $i \in \mathscr{P}$. Applying \ref{lemma: inductive limit, dense subset, extension} with 
    \[
        A = \mathscr{P}, \quad D_i = \mathscr{D}_{K_i}(U, \mathbf{R}^n) \otimes (\mathscr{E}(Y, \mathbf{R}) \cap \mathbf{E}_i(Y)), \quad F_i = \mathbf{H}_{K_i \times \mathbf{B}(0, i)}(U \times Y, \mathbf{R}^n),
    \]
    the assertion follows.
\end{proof}

%----------------------------------------------------------------------------------------------------------

\bibliographystyle{alpha}
\bibliography{ref}

\medskip \noindent \textsc{Affiliation}

\noindent
Department of Mathematics \\
National Taiwan Normal University \\
No.88, Sec.4, Tingzhou Rd. \\
Wenshan Dist., \textsc{Taipei City 116059 \\
	Taiwan (R.O.C.)}

\medskip \noindent \textsc{Email address}

\medskip \noindent
\href{mailto:hsinchuangchou@gmail.com}{hsinchuangchou@gmail.com}

\end{document}